\documentclass{amsart}
\usepackage{amssymb, mathtools, fullpage, color, ulem}
\usepackage[colorinlistoftodos]{todonotes}
\usepackage[colorlinks=true, pdfstartview=FitV,linkcolor=blue,citecolor=blue,urlcolor=blue]{hyperref}
\usepackage{bbm}

\theoremstyle{definition}
\newtheorem {theorem}{Theorem}[section]
\newtheorem {lemma}[theorem]{Lemma}
\newtheorem {corollary}[theorem]{Corollary}

\newtheorem{remark}[theorem]{Remark}
\newtheorem{example}[theorem]{Example}

\newenvironment{red}{\relax\color{red}}{\relax}
\newenvironment{blue}{\relax\color{blue}}{\hspace*{.5ex}\relax}
\newenvironment{jaune}{\relax\color{magenta}}{\relax}
\newcommand{\ber}{\begin{red}}
\newcommand{\er}{\end{red}}
\newcommand{\beb}{\begin{blue}}
\newcommand{\eb}{\end{blue}}
\newcommand{\bej}{\begin{jaune}}
\newcommand{\ej}{\end{jaune}}

\newcommand{\prceil}{\rceil^{\mathfrak p}}
\newcommand{\seteq}{\coloneqq}
\numberwithin{equation}{section}

\begin{document}

\title[Murmurations of Dirichlet Characters]{Murmurations of Dirichlet Characters}

\date{\today}

\author[K.-H. Lee]{Kyu-Hwan Lee$^{\star}$}
\thanks{$^{\star}$This work was partially supported by a grant from the Simons Foundation (\#712100).}
\address{Department of Mathematics, University of Connecticut, Storrs, CT 06269, U.S.A. \hfill \break \indent Korea Institute for Advanced Study, Seoul 02455, Republic of Korea}
\email{khlee@math.uconn.edu}

\author[T. Oliver]{Thomas Oliver}
\address{University of Westminster, London, U.K.}
\email{T.Oliver@westminster.ac.uk}

\author[A. Pozdnyakov]{Alexey Pozdnyakov}
\address{Department of Mathematics, University of Connecticut, Storrs, CT 06269, U.S.A.}
\email{alexey.pozdnyakov@uconn.edu}

\begin{abstract}
We calculate murmuration densities for two families of Dirichlet characters. 
The first family contains complex Dirichlet characters normalized by their Gauss sums.
Integrating the first density over a geometric interval yields a murmuration function compatible with experimental observations.
The second family contains real Dirichlet characters weighted by a smooth function with compact support.
We show that the second density exhibits a universality property analogous to Zubrilina's density for holomorphic newforms, and it interpolates the phase transition in the the $1$-level density for a symplectic family of $L$-functions.
\end{abstract}

\maketitle

\section{Introduction}\label{s:intro}
Following a programme of machine learning in arithmetic \cite{HLO1,HLO2,HLO3}, a striking oscillation in the average value of Frobenius traces for certain families of elliptic curves was observed in \cite{HLOP}. 
This oscillation was termed a \textsl{murmuration}.
In the original work, the average was taken over elliptic curves $E/\mathbb{Q}$ with conductor in certain intervals. 
Similar averages for other arithmetic families, including higher weight modular forms and higher genus curves, will be explored in \cite{HLOPS}.   

After the initial observation, three important ideas emerged based on contributions of J. Ellenberg, A. Sutherland, J. Bober, and P. Sarnak.
Firstly, on Ellenberg's suggestion, Sutherland studied murmurations attached not only to newforms with rational coefficients, but, moreover, Galois orbits of those with coefficients in number fields. 
Secondly, Bober proposed a so-called \textsl{local average}, which eliminated the role played by the interval from the original construction.
Thirdly, Sarnak introduced a notion of \textsl{murmuration density}, which involved additional averaging over primes and weighting by smooth functions of compact support \cite{S23i}.
To motivate his construction, Sarnak articulated the relationship between murmurations and the $1$-level densities for families of $L$-functions (see \cite{S23ii}).
All three ideas informed the important work of N.~Zubrilina, in which a murmuration density for holomorphic newforms was calculated \cite{Z23}.
In this paper we calculate murmuration densities for two families of Dirichlet characters, both of which come from averaging characters over primes in short intervals (see Examples~\ref{ex.tpaverage} and \ref{ex.doubleaverage}).

The first murmuration density we compute involves odd (resp. even) complex Dirichlet characters $\chi$ normalized by their Gauss sums $\tau(\chi)$.
By way of justification, note that  ${\chi}(p)/\tau(\chi)$ is the Fourier coefficient of $\overline{\chi}$ when expanded in terms of additive characters (see, e.g., \cite[equation~(3.12)]{IK}), and so this is a natural analogue of the modular form case. 
Integrating the murmuration density over a given geometric interval yields
the average value of $\chi(p)/\tau(\chi)$ over odd (resp. even) Dirichlet characters with conductor in that interval, which is a scale invariant oscillation comparable with the murmuration first observed for elliptic curves (see Figure~\ref{fig:idft_dyadic_avg}). 
More precisely, we let $\mathcal{D}_{+}(N)$ (resp. $\mathcal{D}_-(N)$) denote the set of primitive even (resp. odd) Dirichlet characters mod $N$. 
For $x\in\mathbb{R}_{>0}$, denote by $\lceil x \prceil$ the smallest prime that is bigger than or equal to $x$. 
For
$c\in\mathbb{R}_{>1}$, $\delta \in (0,1)$, and $y\in\mathbb{R}_{>0}$, define
\begin{align} 
P_\pm (y, X, c)& = \frac{\log X}{X}\sum_{\substack{N \in [X, cX] \\ N \text{ prime}}} \sum_{\chi \in \mathcal{D}_\pm(N)} \frac{\chi(\lceil yX \prceil)}{\tau(\chi)}, \label{eq.P} \\ 
\widetilde P_\pm (y, X, \delta)& =  \frac{\log X}{X^{\delta}}\sum_{\substack{N \in [X, X+X^\delta] \\ N \text{ prime}}} \sum_{\chi \in \mathcal{D}_\pm(N)} \frac{\chi(\lceil yX \prceil)}{\tau(\chi)}.\label{eq.Pt} \end{align}
We plot instances of the functions $P_{\pm}(y,X,c)$ and $\widetilde P_{\pm}(y,X,\delta)$ in Figure~\ref{fig:prime_lavg}.
The factors $\log(X)/X$ and $\log(X)/X^{\delta}$ are connected to the number of primes in the respective intervals.
In the case of equation~\eqref{eq.Pt}, we work conditional on the Riemann hypothesis, which guarantees that the interval $[X,X+X^{\delta}]$ contains primes provided that $\delta>\frac12$.
Our first theorem is stated as follows:

\begin{theorem}\label{thm.main}
Fix $y\in\mathbb{R}_{>0}$.  If $c\in\mathbb{R}_{>1}$, 
then 
\begin{equation}\label{eq.Athm-1}
\lim_{X \to \infty}   P_\pm (y, X, c) = \begin{cases}
\int_1^c \cos \left(\frac{2\pi y}{x}\right)dx,  & \text{ if }+,\\
-i\int_1^c \sin\left(\frac{2\pi y}{x}\right)dx, & \text{ if }-,
\end{cases} 
\end{equation}
and, assuming the Riemann hypothesis, if $\delta \in (\frac12,1)$, then
\begin{equation}
\label{eq.LAthm-1}
\lim_{X \to \infty} \widetilde P_\pm (y, X, \delta) = \begin{cases}
\cos(2\pi y), & \text{ if $+$},\\
-i\sin(2\pi y), & \text{ if $-$}.
\end{cases}
\end{equation}
\end{theorem}

The proof of Theorem~\ref{thm.main} uses the prime number theorem, and the relationship between additive and multiplicative characters.
The fit for $P_{\pm}(y,X,c)$ and $\widetilde P_{\pm}(y,X,\delta)$ given by Theorem~\ref{thm.main} is depicted in Figure~\ref{fig:prime_lavg}, in which we have used the relatively small value $X=2^{10}$.
For small values of $X$, the fit given by Theorem~\ref{thm.main} is far from perfect.
Upon closer inspection, the proof of Theorem~\ref{thm.main} indicates that equation~\eqref{eq.LAthm-1} may be reformulated to incorporate certain composite conductors, and this yields a better fit even for relatively small values of $X$ (cf. Figure~\ref{fig:idft_dyadic_lavg}).
We specify this reformulation in Section~\ref{s:generalconductors}, and furthermore establish variants of Theorem~\ref{thm.main} for arbitrary conductors (in which case we no longer need to assume the Riemann hypothesis).

\medskip

The second murmuration density we compute is the more challenging case of real Dirichlet characters.
In this case, the average value of the Fourier coefficients for those with conductor in a geometric interval yields a noisy image (see Figure~\ref{fig:idft_real_dyadic_sum}).
To counteract this, we use techniques originally developed by Katz--Sarnak and refined by Soundararajan \cite{S00}. 
For $d\in\mathbb{Z}$, we use the notation $\chi_d = \left(\frac{d}{\cdot}\right)$. We also let $\mathcal{G}$ be the set of odd squarefree integers. 
If $d \in \mathcal{G}$ and $d > 0$ (resp. $d < 0$), then $\chi_{8d}$ is an even (resp. odd) primitive real character of conductor $8d$. 
For a smooth Schwartz function $\Phi \geq 0$ with compact support, define
\begin{equation}\label{eq.Mp-1}
M_{\Phi}(y,X,\delta)=\frac{\log X}{X^{1+\delta}}\sum_{\substack{p\in[yX,yX+X^{\delta}] \\ p \text{ prime}}}\sum_{\substack{d\in\mathcal{G}}}\Phi\left(\frac{d}{X}\right)\chi_{8d}(p)\sqrt{p}.
\end{equation}
Notice that one can isolate even (resp. odd) characters in this sum by choosing $\Phi$ to have compact support in $\mathbb{R}_{> 0}$ (resp. $\mathbb{R}_{< 0}$). 
In Figure \ref{fig:quad_murm}, we plot equation~\eqref{eq.Mp-1} for two choices of $\Phi$. 

\begin{theorem}\label{thm.main-real}
Fix $y\in\mathbb{R}_{>0}$.
If $\delta\in(\frac34,1)$ and $\Phi \geq 0$ is a smooth Schwartz function with compact support, then, assuming the Generalized Riemann hypothesis, we have
\begin{equation}\label{eq.limreal}
M_\Phi (y, \delta) \seteq \lim_{X\rightarrow\infty}M_{\Phi}(y,X,\delta) =  \frac{1}{2}\sum_{\substack{a =1 \\ (a, 2) = 1}}^\infty \frac{\mu(a)}{a^2} \sum_{m =1}^\infty (-1)^m \widetilde{\Phi}\left(\frac{m^2}{2a^2y} \right),
\end{equation}
where
\begin{equation}\label{eq.Stransform}
\widetilde{\Phi}(\xi)= \int_{-\infty}^\infty \left(\cos(2\pi \xi x)+ \sin(2\pi \xi x)\right)\Phi(x)dx.
\end{equation}
\end{theorem}
We note that, in Figure~\ref{fig:quad_murm}, and the related Figures~\ref{fig:idft_dyadic_avg} and~\ref{fig:quad_murm-1}, we use the value $\delta=\frac23$, which is smaller than the minimal $\delta$ included in Theorem~\ref{thm.main-real}. 
These figures offer some evidence that Theorem~\ref{thm.main-real} may remain valid for such values of $\delta$. 
Furthermore, although the smoothness of $\Phi$ is crucial in enabling the analytic tools used in the proof, we expect Theorem~\ref{thm.main-real} to hold for weight functions with a sharp cut-off as well (we refer to Figure \ref{fig:sharp} for numerical support of this claim).
We remark that a murmuration density for a family of real Dirichlet characters was first computed by Rubinstein and Sarnak in \cite[equation~(13)]{S23ii}, although their formulation and evaluation is different to ours in places. 
Rubinstein and Sarnak also noted that the murmuration density interpolates the phase transition for the $1$-level density of a symplectic family,
which emerges from our analysis in the following form.
\begin{corollary}\label{c:1ld}
Let $\Phi \ge 0$ be a Schwartz function with compact support and let $\delta \in (\frac34, 1). $ 
Assuming the Generalized Riemann hypothesis, we have
\begin{equation}\label{eq.1leveldensity1}
\lim_{y\rightarrow0^+} M_\Phi(y, \delta)
=0, \quad \text{ and } \quad   \lim_{y\rightarrow \infty} M_{\Phi}(y, \delta) =-\frac 2 {\pi^2} \widetilde \Phi(0),
\end{equation}
where $M_\Phi(y,\delta)$ is defined in equation~\eqref{eq.limreal}.
\end{corollary}
We present a proof of Corollary \ref{c:1ld} in Section~\ref{s:1ld}, in which we use the same techniques as Rubinstein and Sarnak.  
A similar phenomenon for real character sums was previously observed in \cite{CFS}, which studied the asymptotics for double sums of the form
\[\sum_{\substack{m\leq X\\m\text{ odd}}}\sum_{\substack{n\leq Y\\n\text{ odd}}}\left(\frac{m}{n}\right).\]
In \cite{CFS}, the authors study the case that $X\sim Y$, which yields a function exhibiting murmuration-like properties, including scale-invariance and non-differentiablity.
The analysis presented in \cite{CFS} is different from that presented here. 
 
The proof of Theorem~\ref{thm.main-real} involves identities for the M\"obius function, the Polya--Vinogradov inequality for sums over primes, and Poisson summation as in \cite[Lemma~2.6]{S00}.
The transform in equation~\eqref{eq.Stransform} was used in \cite[Section~2.4]{S00}.
Unfolding this transform and applying the identity $\cos(x)+\sin(x) = \sqrt{2}\cos(x-\pi/4)$ to equation~\eqref{eq.Stransform} we conclude that
\begin{align}\label{eq.unfolding}
\begin{split}
    \lim_{X \to \infty} M_{\Phi}(y,X,\delta) &= \frac{1}{2}\sum_{\substack{a =1 \\ (a, 2) = 1}}^\infty \frac{\mu(a)}{a^2} \sum_{m =1}^\infty  (-1)^m \int_{-\infty}^\infty \left(\cos\left( \frac{\pi m^2x}{a^2y} \right)+ \sin\left(\frac{\pi m^2 x}{a^2y} 
    \right)\right)\Phi(x)dx \\ &=\int_{-\infty}^\infty\Phi(x)\left(\frac{\sqrt{2}}{2} \sum_{\substack{a =1 \\ (a, 2) = 1}}^\infty   \frac{\mu(a)}{a^2}\sum_{m =1}^\infty  (-1)^m \cos \left(\frac{\pi m^2 x}{a^2y} - \frac{\pi}{4} \right)\right)dx.
\end{split}    
\end{align}
In other words, conditional on the Generalized Riemann hypothesis, we exhibit a distribution $M$ such that, for every smooth Schwartz $\Phi\geq0$ with compact support and every $\delta\in(\frac34,1)$, we have
\begin{equation}\label{eq.ZubDensity}
\lim_{X\rightarrow\infty}M_{\Phi}(y,X,\delta)=\int_{-\infty}^{\infty}\Phi(t)M(y/t)dt.
\end{equation}
Consequently, using Sarnak's terminology, the distribution $M$ in equation~\eqref{eq.ZubDensity} is the \textsl{Zubrilina density} for the family $\left\{\left(\frac{8d}{\cdot}\right):d\in\mathcal{G}\right\}$ \cite{S23ii}.
Using the same techniques, we calculate the Zubrilina density for $\left\{\left(\frac{d}{\cdot}\right):d\in\mathcal{G}\right\}$ in Section~\ref{sec-zd} and deduce the analogue of Corollary~\ref{c:1ld}.
Given Figure \ref{fig:sharp}, it is not immediately clear whether or not $M_{\Phi}(y,X,\delta)$ exhibits infinitely many sign changes near $y=0$.
Zubrilina has shown that, in the setting of modular forms, the analogous function has only finitely many sign changes.

\medskip

We conclude this introduction with a summary of the sequel. 
Section~\ref{s:background} contains the relevant background material on Dirichlet characters.
In Section~\ref{s:prime}, we prove Theorem~\ref{thm.main}. 
In Section~\ref{s:realp}, we prove Theorem~\ref{thm.main-real}.
In Section~\ref{s:1ld}, we prove Corollary~\ref{c:1ld}. 
In Section~\ref{s:generalconductors}, we state and prove the aforementioned variations on Theorems \ref{thm.main} and~\ref{thm.main-real} (both of which concern averaging over an alternative set of conductors).

\begin{figure}[h]
\centering
\includegraphics[width=0.7\textwidth]{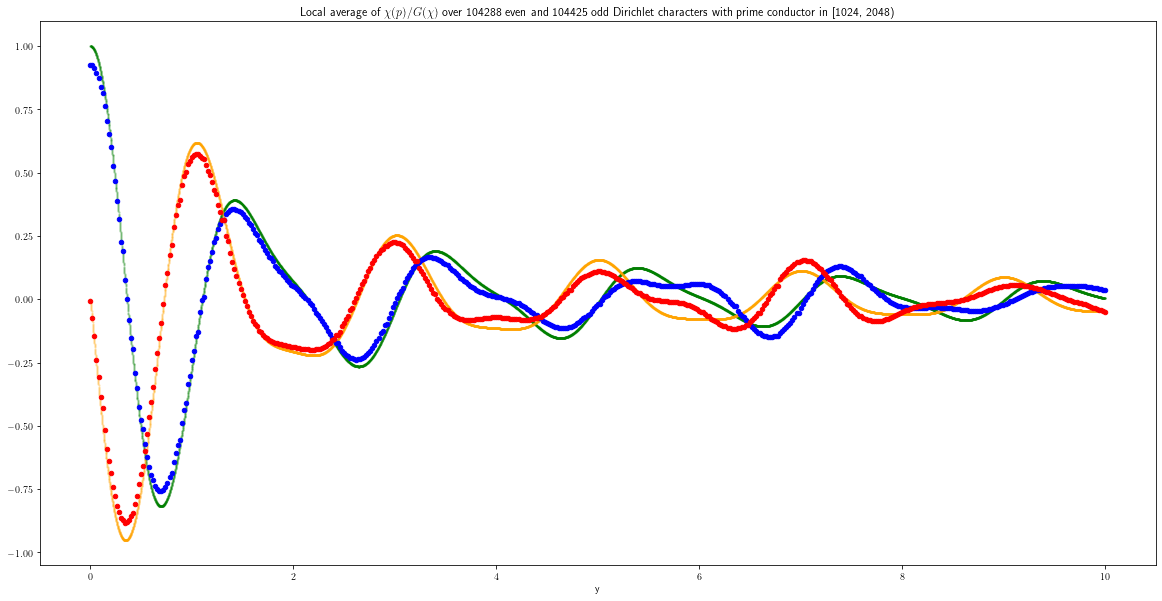}
\includegraphics[width=0.7\textwidth]{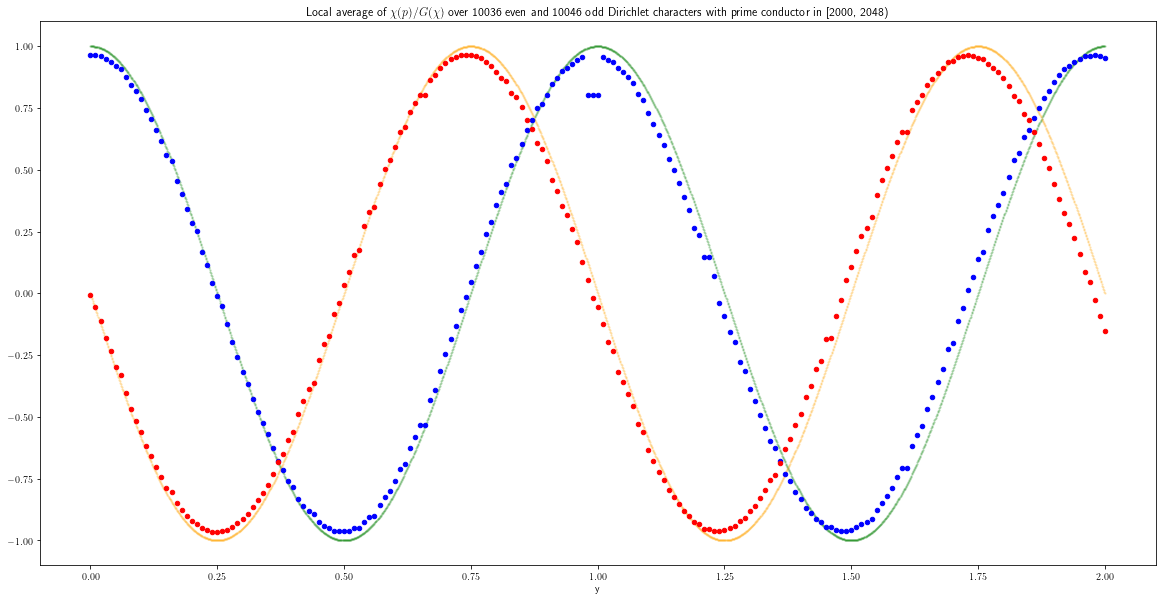}
\caption{\sf (Top) $P_\pm(y, 2^{10}, 2)$ for $y \in [0,10]$ with $+$ in blue and (the imaginary part of) $-$ in red. (Bottom) $\widetilde{P}_\pm(y, 2002, 0.51)$ for $y \in [0,2]$ with $+$ in blue and (the imaginary part of) $-$ in red. The solid curves (in yellow and green) represent the limits given by Theorem~\ref{thm.main}. The discontinuity around $y=1$ will be explained in Remark~\ref{rem.discont}. }
    \label{fig:prime_lavg}
\end{figure}

\begin{figure}[h]
\centering
\includegraphics[width=0.8\textwidth]{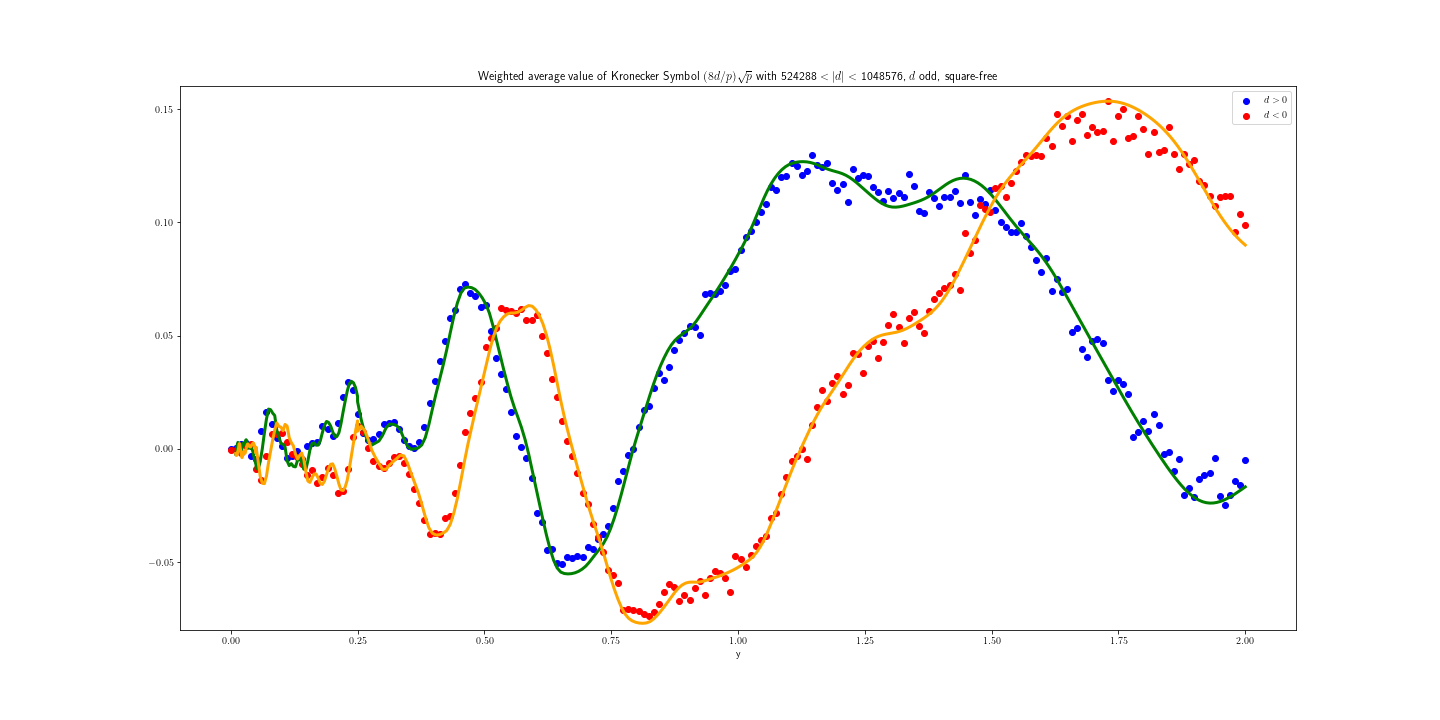}
\caption{\sf Let \[\Phi_+(x) =  \mathbbm{1}_{(1,2)}(x)\exp  \left(\tfrac{-1}{1-4(x-1.5)^2}  \right), \] \[ \Phi_-(x) =  \mathbbm{1}_{(-2,-1)}(x)\exp \left(\tfrac{-1}{1-4(-x-1.5)^2} \right).\]
We plot $M_{\Phi_\pm}(y, 2^{19}, \frac23)$ for $y \in [0,2]$ with $+$ in blue (resp. $-$ in red).
We also plot the right hand side of equation~\eqref{eq.limreal} in green (resp. orange). 
}
\label{fig:quad_murm}
\end{figure}

\begin{figure}[h]
\centering
\includegraphics[width=0.8\textwidth]{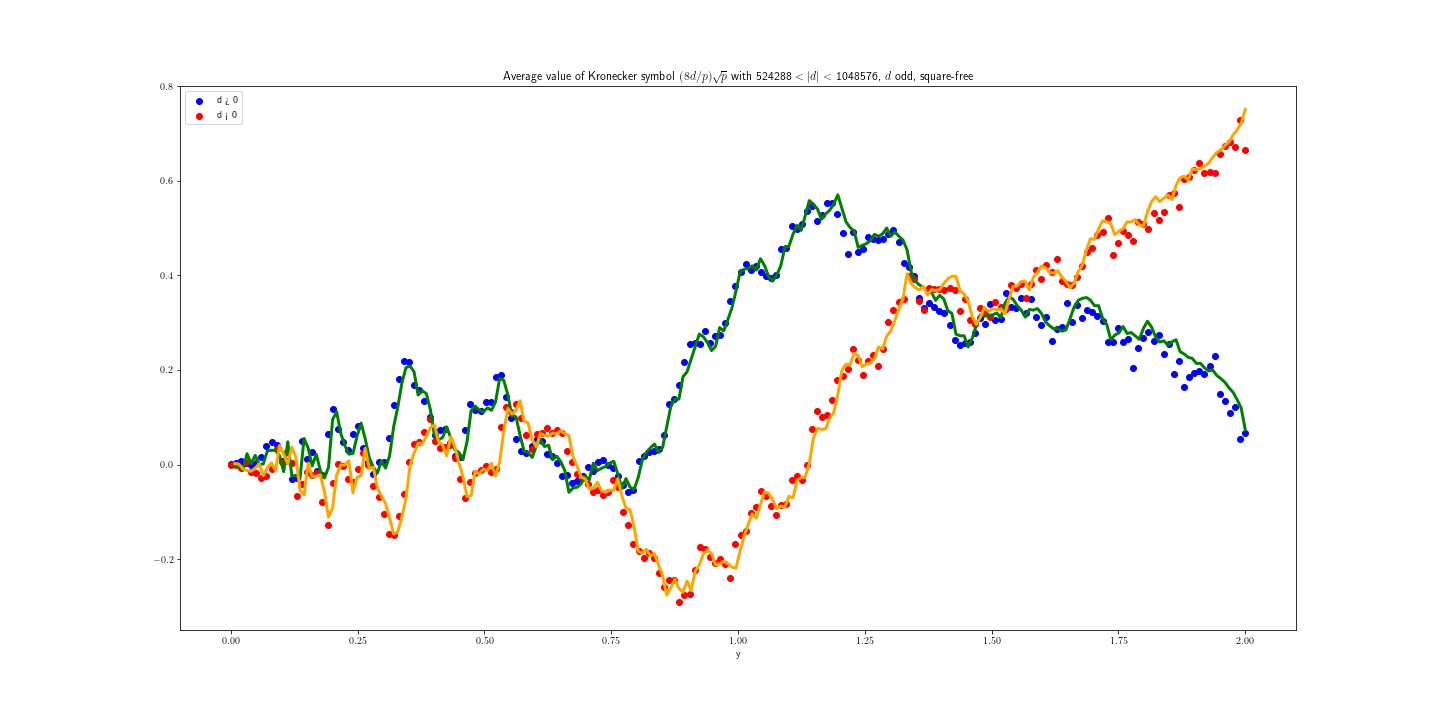}
\caption{\sf 
Let $\Phi_+(x) =  \mathbbm{1}_{(1,2)}(x)$ and $\Phi_-(x) =  \mathbbm{1}_{(-2,-1)}(x)$. We plot $M_{\Phi_+}(y, 2^{19}, \frac23)$ (resp. $M_{\Phi_-}(y, 2^{19}, \frac23)$) in blue (resp. red). We also plot the right hand side of equation~\eqref{eq.limreal} in green (resp. orange). }
\label{fig:sharp}
\end{figure}

\subsection*{Acknowledgements}
The authors are grateful to Yang-Hui He and Andrew Sutherland for preliminary conversations connected to the themes of this paper, to Kumar Murty for helpful comments on an early draft, to Peter Sarnak for suggesting the use of Poisson summation and several insightful discussions, to Philip Holdridge for their useful comments on an earlier draft, to Kannan Soundararajan for drawing the authors' attention to \cite{CFS}, and to the anonymous referees for recommending that we expand an earlier draft to include the family of real Dirichlet characters and for facilitating several improvements with their careful reading and detailed comments.

\section{Background}\label{s:background}
\subsection{Asymptotics of double averages}
For $m\in\mathbb{Z}_{>0}$, a Dirichlet character mod $m$ is a completely multiplicative function $\chi:\mathbb{Z}\rightarrow\mathbb{C}$ which is periodic with period $m$ and satisfies $\chi(a)=0$ if and only if $\mathrm{gcd}(a,m)>1$. 
The Gauss sum of a Dirichlet character $\chi$ mod $m$ is defined by
\[\tau(\chi)=\sum_{b=1}^m\chi(b)e^{2\pi ib/m}.\]
We denote by $\chi_0$ the principal Dirichlet character mod $m$, which satisfies $\chi_0(a)=1$ for $(a,m)=1$ by definition.
We say that a Dirichlet character $\chi$ is even (resp. odd) if $\chi(-1)=1$ (resp. $\chi(-1)=-1$).
The conductor of a Dirichlet character $\chi$ is the minimal positive integer $N$ such that $\chi$ is a Dirichlet character mod $N$.
We say that a Dirichlet character $\chi$ is primitive if its modulus and conductor are equal. 
We let $\mathcal{D}_{+}(N)$ (resp. $\mathcal{D}_-(N)$) denote the set of primitive even (resp. odd) Dirichlet characters mod $N$. 
A Dirichlet character is said to be quadratic if its values are real.
We denote by $\mathcal{Q}_\pm(N)$ the subset of $\mathcal{D}_{\pm}(N)$ consisting of quadratic characters. 
Note that, for even (resp. odd) characters $\chi\in \mathcal{Q}_\pm(N)$, we have $\tau(\chi)= \sqrt N$ (resp. $i\sqrt N$).

\begin{example}\label{ex.quadratic}
Quadratic characters provide the simplest analogue to the murmurations of elliptic curves over $\mathbb{Q}$ discovered in \cite{HLOP}.
Furthermore, using quadratic reciprocity, one may relate sums of quadratic Dirichlet characters to Chebyshev's bias (cf. \cite{RS}).
In Figure~\ref{fig:idft_real_dyadic_sum}, we plot the sum of $\chi(p)/\tau(\chi)$ over $\bigcup_{N=X}^{2X-1} \mathcal{Q}_\pm(N)$ for $X=2^{17}$. 
\end{example}
\begin{figure}[h]
\centering
\includegraphics[width=0.7\textwidth]{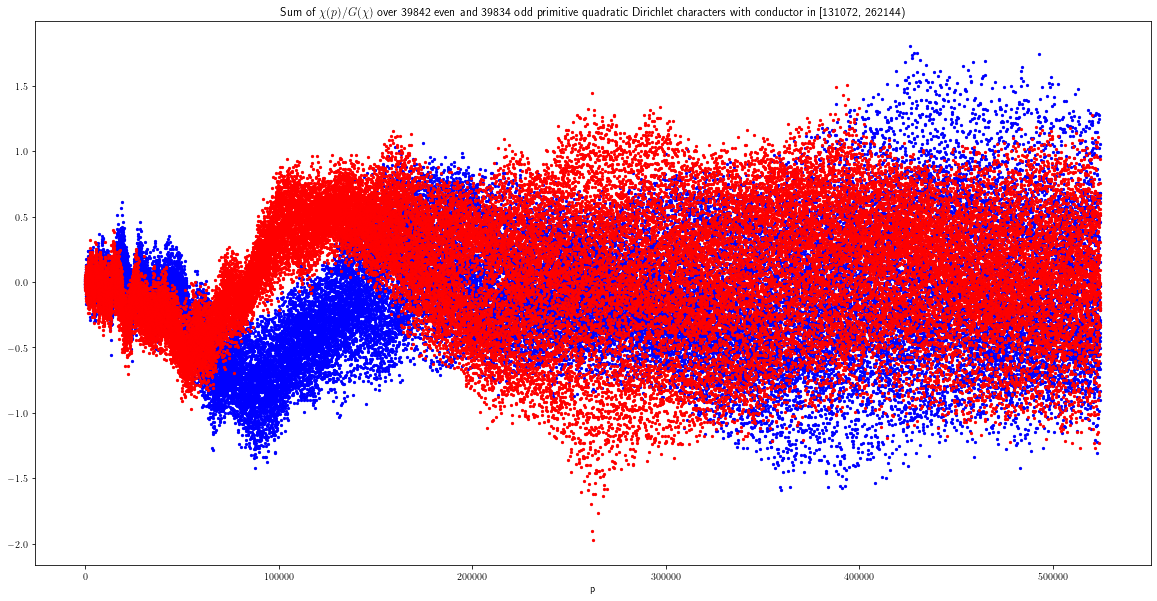}
\caption{\sf Plot of $\sum_{N \in [X, 2X)}\sum_{\chi \in \mathcal{Q}_\pm(N)} \chi(p)/\tau(\chi)$, for $X = 2^{17}$ and $2 \leq p < 4X$ with $+$ in blue and (the imaginary part of) $-$ in red.}\label{fig:idft_real_dyadic_sum}
\end{figure}

In this paper, we consider two variations of the sum considered in Figure~\ref{fig:idft_real_dyadic_sum}.
The first, and simplest, variation is to involve (Galois orbits of) complex characters in the average.
The second, more challenging, variation is to work only with real characters but to incorporate a smooth weight function with compact support and to take the average over the primes in  a short interval.  

\begin{example}\label{ex.orbits}
In Figure~\ref{fig:idft_dyadic_avg}, we plot the sum of $\chi(p)/\tau(\chi)$ over $\chi\in\bigcup_{N=X}^{2X-1} \mathcal{D}_\pm(N)$, for $X=2^{10}$,
normalized by (cf. \cite{J73}):
\begin{equation}\label{eq.chipi}
\frac{1}{X}\sim\frac{3\sqrt{3}}{\pi^2\sqrt{\# \mathcal D_{\pm}(X)}} 
\end{equation}
We note that including the non-real characters and normalizing in this way yields a much less noisy image than in Figure~\ref{fig:idft_real_dyadic_sum}.
We will observe a similar effect with modular forms in a forthcoming work \cite{HLOPS}.
\end{example}

\begin{figure}[h]
\includegraphics[width=0.6\textwidth]{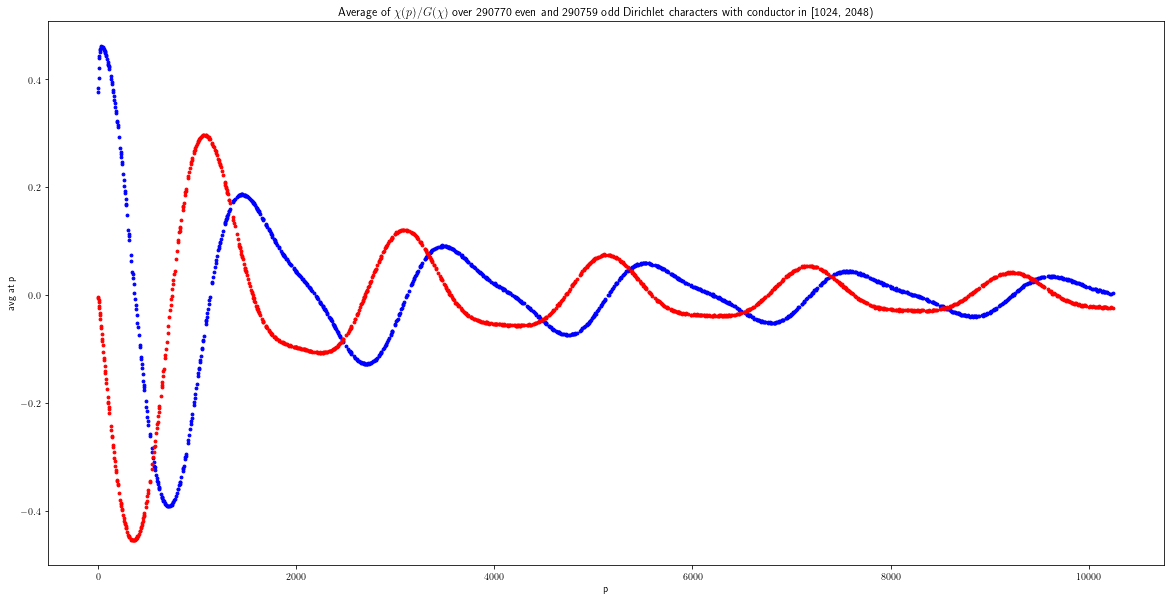}
\caption{\sf Plot of $\frac{1}{X}\sum_{N \in [X, 2X)}\sum_{\chi \in \mathcal{D}_\pm(N)} \chi(p)/\tau(\chi)$ for $X = 2^{10}$ for primes $p$ such that $2 \leq p \leq 10 X$, with $+$ in blue and (the imaginary part of) $-$ in red. }
\label{fig:idft_dyadic_avg}
\end{figure}

\begin{example}\label{ex.tpaverage}
The function $\widetilde{P}_{\pm}(y,X,\delta)$ in equation~\eqref{eq.Pt} comes from the following average:
\begin{equation}\label{eq.exp}
\frac{\sum_{\substack{N\in[X,X+X^{\delta}]\\N\text{ prime}}}\sum_{\chi\in\mathcal{D}_{\pm}(N)}\chi(\lceil yX\rceil^{\mathfrak{p}})/ \tau(\chi)}{\sum_{\substack{N\in[X,X+X^{\delta}]\\N\text{ prime}}}\, 1},
\end{equation}
where $y\in\mathbb{R}_{>0}$ and $\delta\in(\frac12,1)$.
Indeed, assuming the Riemann hypothesis and applying the prime number theorem, we deduce 
\begin{equation}
\lim_{X\rightarrow\infty}\left[\#\left\{N\in[X,X+X^{\delta}]:N~\text{prime}\right\} \cdot \frac{\log(X)}{X^{\delta}}\right]=1.
\end{equation}
It follows that the function in equation~\eqref{eq.exp} is asymptotic to $\widetilde{P}_{\pm}(y,X,\delta)$.
There is a similar interpretation for the equation $P_{\pm}(y,X,c)$ in equation~\eqref{eq.P}.
\end{example}

\begin{example}\label{ex.doubleaverage}
The function $M_{\Phi}(y,X,\delta)$ in equation~\eqref{eq.Mp-1} comes from the following double average:
\begin{equation}\label{eq.doubleexp}
D_{\Phi}(y,X,\delta)=\frac{\sum_{\substack{p\in[yX,yX+X^{\delta}] \\ p \text{ prime}}}\frac{\sum_{d\in\mathcal{G}}\Phi(d/X)\chi_{8d}(p)\sqrt{p}}{\sum_{d\in\mathcal{G}}\Phi(d/X)}}{\sum_{\substack{p\in[yX,yX+X^{\delta}] \\ p \text{ prime}}}\, 1},
\end{equation}
where $y\in\mathbb{R}_{>0}$, $\delta\in(\frac34,1)$, $\Phi$ is a smooth function of compact support, $\mathcal{G}$ denotes the set of odd squarefree integers, and $\chi_d$ denotes the Kronecker symbol $\left(\frac{d}{\cdot}\right)$.
Assuming the Riemann hypothesis and applying the prime number theorem, we deduce 
\begin{equation}\label{eq.PNT}
\lim_{X\rightarrow\infty}\left[\#\left\{p\in[yX,yX+X^{\delta}]:p~\text{prime}\right\} \cdot \frac{\log(X)}{X^{\delta}}\right]=1.
\end{equation}
It follows that $D_{\Phi}(y,X,\delta)$ is asymptotic to
\begin{equation}\label{eq.pntde}
\frac{\log X}{X^{\delta}}\sum_{\substack{p\in[yX,yX+X^{\delta}] \\ p \text{ prime}}}\frac{\sum_{d\in\mathcal{G}}\Phi(d/X)\chi_{8d}(p)\sqrt{p}}{\sum_{d\in\mathcal{G}}\Phi(d/X)}.
\end{equation}
To simplify the denominator in equation~\eqref{eq.pntde}, we note that the natural density of $\mathcal{G}$ is shown to be $4/\pi^2$ in \cite{J10}. Using the fact that $\Phi$ is Schwartz, an equidistribution argument for $(d/X)_{d \in \mathcal{G}}$ implies that
\begin{equation}\label{eq.si}
\lim_{X \to \infty} \frac 1 X \sum_{d\in\mathcal{G}}\Phi(d/X)= \frac{4}{\pi^2} ~\int_{-\infty}^{\infty}\Phi(\tau)d\tau < \infty.
\end{equation} 
Therefore, to understand asymptotics of $D_{\Phi}(y,X,\delta)$, it suffices to analyse the limit of $M_{\Phi}(y,X,\delta)$.
\end{example}

\subsection{Lemmas for Theorem~\ref{thm.main}}\label{s:background-1.1}
We begin with the following Lemma on Gauss sums.
\begin{lemma} \label{lem.cosi}
Let $N$ be a positive integer.
If $p$ is a prime such that $(p,N)=1$, then 
\begin{align}
\cos \left(\frac{2 \pi p}N \right) &= \frac{-1}{\phi(N)}   + \frac{1}{\phi(N)} \sum_{\substack{\chi \bmod N \\ \ \chi \neq \chi_0, \, \chi(-1) = 1}} \tau(\overline{\chi}) \chi(p), \label{eq.cos-1}\\
\sin\left(\frac{2 \pi p}{N} \right) &= \frac{-i}{\phi(N)} \sum_{\substack{\chi \bmod N \\ \chi(-1) = -1}} \tau(\overline{\chi}) \chi(p).\label{eq.sin-1}
\end{align}
\end{lemma}

\begin{proof}
This follows from \cite[(3.11)]{IK}.
\end{proof}

If $N$ is prime, then every non-trivial Dirichlet character mod $N$ is primitive  and hence
\begin{equation}\label{eq.Dpm}
\mathcal{D}_{+}(N) = \{\chi~\mathrm{mod}~N~:~\chi\neq\chi_0,\, \chi(-1) = 1 \},
\ \ 
\mathcal{D}_{-}(N) = \{\chi~\mathrm{mod}~N~:~ \chi(-1)=-1 \} \ \ (N~\text{prime}).
\end{equation}

\begin{lemma}\label{lem.sCGcossin}
If $p$ and $N$ are two distinct primes, then
\begin{align}
\sum_{\chi \in \mathcal{D}_+(N)} \frac{\chi(p)}{\tau(\chi)} &= \left(\frac{N-1}{N} \right)\cos\left(\frac{2\pi p}{N}\right)+\frac{1}{N},\label{eq.mcos}\\
\sum_{\chi \in \mathcal{D}_-(N)} \frac{\chi(p)}{\tau(\chi)} &= -i\left(\frac{N-1}{N}\right)\sin\left(\frac{2 \pi p}{N}\right).\label{eq.msin}
\end{align}
\end{lemma}
\begin{proof}
For $\chi \in \mathcal{D}_\pm(N)$, recall 
\begin{equation}\label{eq.1/G}
\frac 1{\tau(\chi)} = \frac {\chi(-1)} N \tau(\overline{\chi}) ,
\end{equation}
(see, for example, \cite[Exercise 1.1.1]{Bump}). 
Since $\epsilon \seteq \chi(-1)$ is constant on $\chi \in \mathcal{D}_\pm(N)$, summing equation~\eqref{eq.1/G} over $\chi\in\mathcal{D}_{\pm}(N)$ yields
\begin{align}\label{eq.murm-1Gbarchi}
\sum_{\chi \in \mathcal{D}_\pm(N)} \frac{\chi(p)}{\tau(\chi)} = 
\frac{\epsilon}{N} \sum_{\chi \in \mathcal{D}_\pm(N)} \tau(\overline{\chi})\chi(p).
\end{align}
Since $N$ is prime and $(p, N)=1$, Lemma \ref{lem.cosi} implies
\begin{align}
\cos\left(\frac{2 \pi p}{N} \right) &=  \frac{-1}{N-1} + \frac{1}{N-1} \sum_{\substack{\chi \bmod N \\ \chi \neq \chi_0,  \chi(-1) = 1}} \tau(\overline{\chi}) \chi(p), \label{eq.cos}\\
\sin\left(\frac{2 \pi p}{N} \right) &= \frac{-i}{N-1} \sum_{\substack{\chi \bmod N \\ \chi(-1) = -1}} \tau(\overline{\chi}) \chi(p).\label{eq.sin}
\end{align}
The result now follows from equations~\eqref{eq.Dpm},~\eqref{eq.murm-1Gbarchi},~\eqref{eq.cos} and ~\eqref{eq.sin}.
\end{proof}

\begin{lemma} For $a \in \mathbb{R}_{>0}$ and $b \in (0,1]$, we have
\begin{equation} \label{eq.1/N}
    \lim_{X \to \infty} \frac{\log X}{X^b} \sum_{\substack{N \in [X, X+aX^b]}} \frac{1}{N} = 0.
\end{equation}
\end{lemma}
\begin{proof} Since $a,b, N$ are all positive, we have
\begin{equation*}
    \lim_{X \to \infty} \frac{\log X}{X^b} \sum_{\substack{N \in [X, X+aX^b]}} \frac{1}{N} \leq \lim_{X \to \infty} \frac{\log X}{X^b} \sum_{0 < N \leq (a+1)X} \frac{1}{N} = \lim_{X \to \infty} O\left(\frac{\log X \log ((a+1)X)}{X^b} \right) = 0.
\end{equation*}
\end{proof}

\begin{lemma}
For $y\in\mathbb{R}_{>0}$, if $N \geq X$, we have
\begin{equation} \label{eq.yx} 
\lim_{X \to \infty} \frac{ \lceil yX \prceil - yX}N =0. \end{equation} 
\end{lemma}

\begin{proof}
For any $x\in\mathbb{R}_{>0}$, we have $\lceil x\prceil-x<x^\theta$ for some constant $\theta <1$ (see, e.g., \cite{BHP}). Subsequently, we deduce that
\[\lim_{X \to \infty} \frac{ \lceil yX \prceil - yX}N \le \lim_{X \to \infty} \frac{ \lceil yX \prceil - yX}X=0. \]
\end{proof}

\begin{lemma}\label{lem:often-used}
Fix $\eta \in \mathbb{R}_{> 0}$ and $\delta \in (\frac12, 1)$. If $f : \mathbb{R} \to \mathbb{C}$ is continuous, then, assuming the Riemann hypothesis, we have
    \begin{equation}\label{eq.often-used}
        \lim_{X \to \infty} \frac{\log X}{X^\delta} \sum_{\substack{p\in[\eta X,\eta X+X^{\delta}] \\ p \text{ prime}}} f\left(\frac{p}{X}\right) = f(\eta)
    \end{equation}
\end{lemma}

\begin{proof}
Since we have $p/X \to \eta$ as $X\rightarrow\infty$ for $p\in[\eta X,\eta X+X^{\delta}]$, we know that, for all $\epsilon' > 0$, there exists $X_0$ such that $X > X_0$ implies $|p/X-\eta |<\epsilon'$. 
Since $f$ is continuous, for all $\epsilon > 0$ there exists $\epsilon' > 0$ such that $|p/X - \eta| < \epsilon'$ implies $|f(p/X) - f(\eta)| < \epsilon$. Thus, for $X$ sufficiently large, we have:
    \begin{align}\label{eq.abssum}
        \left\lvert\sum_{\substack{p\in[\eta X,\eta X+X^{\delta}] \\ p \text{ prime}}} f\left(\frac{p}{X}\right) -  f(\eta)\sum_{\substack{p\in[\eta X,\eta X+X^{\delta}] \\ p \text{ prime}}} 1 \right\rvert \leq  \sum_{\substack{p\in[\eta X,\eta X+X^{\delta}] \\ p \text{ prime}}} \left\lvert f\left(\frac{p}{X}\right) - f(\eta )\right\rvert < \epsilon \sum_{\substack{p\in[\eta X,\eta X+X^{\delta}] \\ p \text{ prime}}} 1.
    \end{align}
    Multiplying equation~\eqref{eq.abssum} by $\log X/X^\delta$, and using equation \eqref{eq.PNT}, we deduce that:
    \begin{equation}
        \left\lvert\lim_{X \to \infty} \frac{\log X}{X^\delta}\sum_{\substack{p\in[\eta X,\eta X+X^{\delta}] \\ p \text{ prime}}} f\left(\frac{p}{X}\right) - f(\eta )\right\rvert < \epsilon.
    \end{equation}
    Since $\epsilon > 0$ is arbitrary, we deduce equation~\eqref{eq.often-used}.
\end{proof}

\subsection{Lemmas for Theorem~\ref{thm.main-real}}
We begin with the following manifestation of the Polya--Vinogradov inequality.
\begin{lemma} \label{lem.PV}
Let $y\in\mathbb{R}_{>0}$, and let $d \in \mathbb Z$ be such that $\chi_d$ is non-principal.  
If $\delta \in (\frac12, 1)$  
then assuming the Generalized Riemann hypothesis, for any $\epsilon>0$, as $X\rightarrow\infty$, we have
\begin{equation}\label{eq.PV}
    \left|\sum_{\substack{p\in[yX,yX+X^{\delta}] \\ p \text{ prime}}} \chi_d(p)\right| \ll (yX)^{\frac12+\epsilon}.
\end{equation}
\end{lemma}
\begin{proof}
This follows from \cite[equation (5.1)]{GS07}.
\end{proof}
Following \cite[Section~(2.2)]{S00}, for an integer $k$ and a prime number $p$, we define
\begin{equation}
    G_k(p) = \left(\frac{1 - i}{2} + \left( \frac{-1}{p}\right) \frac{1+i}{2}\right) \sum_{b \bmod p} \left(\frac{b}{p}\right)
    e^{2\pi i bk/p},
\end{equation}
and
\begin{equation}
    \tau_k(p) = \sum_{b \bmod p}\left(\frac{b}{p}\right)e^{2\pi i bk/p}, 
\end{equation}
so that
\begin{equation}\label{eq.tauG}
    \tau_k(p) =  \left(\frac{1 + i}{2} + \left( \frac{-1}{p}\right) \frac{1 -  i}{2}\right) G_k(p).
\end{equation}
Moreover, using the notation from Section~\ref{s:background-1.1}, we have $\tau_1(p)=\tau\left(\left(\frac{\cdot}{p}\right)\right)$.
For a smooth Schwartz function $\Phi$, we let $\widetilde{\Phi}$ be as in equation~\eqref{eq.Stransform}.
At various points in what follows, we will use the fact that, if $\Phi$ is Schwartz, then $\widetilde{\Phi}$ is Schwartz.
We will also use the notation $\widehat{\Phi}$ to denote the usual Fourier transform, that is
\[\widehat{\Phi}(\xi)=\int_{-\infty}^{\infty} \Phi(x)e^{-2\pi i x\xi}dx.\]
Note that
\begin{equation}
    \tau_k(p)\widehat{\Phi}\left(\frac{kX}{\alpha a^2 p} \right) + \tau_{-k}(p)\widehat{\Phi}\left(\frac{-kX}{\alpha a^2 p} \right) = G_k(p)\widetilde{\Phi}\left(\frac{kX}{\alpha a^2 p} \right) + G_{-k}(p)\widetilde{\Phi}\left(\frac{-kX}{\alpha a^2 p} \right). 
\end{equation}
Since $G_k(p) = \left(\frac{-1}{p}\right)G_{-k}(p)$ and $\tau_0(p) = G_0(p) = 0$, equation~\eqref{eq.tauG} implies:
\begin{equation}\label{eq:recombined}
    \frac{X}{\alpha a^2 p} \sum_{k \in \mathbb{Z}} \tau_k(p)\widehat{\Phi}\left( \frac{kX}{\alpha a^2p}\right) = \frac{X}{\alpha a^2 p} \sum_{k \in \mathbb{Z}} G_k(p)\widetilde{\Phi}\left(\frac{kX}{\alpha a^2 p}\right).
\end{equation}
For completeness, we prove the following form of \cite[Lemma~2.6]{S00}.
\begin{lemma}\label{lem.Sound-lemma}
Let $\Phi\geq0$ be a smooth function with compact support and let $\beta = \sup_{x \in \mathbb{R}} \{|x| : \Phi(x) > 0\}$.
For a prime number $p$, and any $A\in (0,\sqrt{\beta X}]$, we have
\begin{equation}\label{eq.SMap}
    \frac{1}{X} \sum_{\substack{d \in \mathbb{Z} \\ (d,2)=1}}\Big(\sum_{\substack{a^2 \mid |d| \\ a \leq A}} \mu(a) \Big) \Phi\left( \frac{d}{X}\right)\left( \frac{d}{p}\right)\sqrt{p} = \frac{1}{2} \left(\frac{2}{p}\right) \sum_{\substack{0< a \leq A \\ (a, 2p) = 1}} \frac{\mu(a)}{a^2} \sum_{\substack{k\in \mathbb{Z}}} (-1)^k \left(\frac{k}{p} \right)\widetilde{\Phi}\left(\frac{kX}{2a^2 p} \right).
\end{equation}
\end{lemma}
\begin{proof}
By switching the order of summation and using $\left(\frac{da^2}{p}\right) = \left(\frac{d}{p}\right)\left(\frac{a}{p}\right)^2$, we deduce that:
\begin{equation}\label{eq:sound-identity}
    \sum_{\substack{d \in \mathbb{Z} \\ (d,2)=1}}\Big(\sum_{\substack{a^2 \mid |d| \\ a \leq A}} \mu(a) \Big) \Phi\left(\frac{d}{X}\right)\left(\frac{d}{p}\right)= \sum_{\substack{a \leq A \\ (a, 2p) = 1}} \mu(a) \sum_{\substack{d \in \mathbb{Z} \\ (d,2)=1}}\Phi\left(\frac{da^2}{X}\right)\left(\frac{d}{p}\right).
\end{equation}
We observe that
\begin{equation}\label{eq:odd-split}
    \sum_{\substack{d \in \mathbb{Z} \\ (d,2)=1}}  \Phi\left(\frac{da^2}{X}\right)\left(\frac{d}{p}\right) = \sum_{d \in \mathbb{Z}} \Phi\left(\frac{da^2}{X}\right)\left(\frac{d}{p}\right) - \left(\frac{2}{p}\right)\sum_{d \in \mathbb{Z}} \Phi\left(\frac{2da^2}{X}\right)\left(\frac{d}{p}\right),
\end{equation}
and, for $\alpha \in \{1,2\}$, we write
\begin{equation}
    \sum_{d \in \mathbb{Z}} \left(\frac{d}{p}\right)\Phi\left(\frac{\alpha d a^2}{X}\right) = \sum_{b \bmod p} \left(\frac{b}{p}\right) \sum_{d \in \mathbb{Z}}  \Phi\left(\frac{\alpha a^2 (pd + b)}{X}\right).
\end{equation}
Poisson summation implies that
\begin{align}\label{eq.PS}
\begin{split}
     \sum_{d \in \mathbb{Z}}  \Phi\left(\frac{\alpha a^2(pd+b)}{X}\right) &= \sum_{k \in \mathbb{Z}} \int_{-\infty}^\infty \Phi\left(\frac{\alpha a^2(p\xi+b)}{X}\right)e\left(-\xi k\right)d\xi   \\
     &= \frac{X}{\alpha a^2 p} \sum_{k \in \mathbb{Z}}  \int_{-\infty}^{\infty} \Phi(u)e\left(\frac{kb}{p} - \frac{kXu}{\alpha a^2 p}\right)du \\
     &= \frac{X}{\alpha a^2 p}\sum_{k \in \mathbb{Z}}  e\left(\frac{kb}{p}\right)  \int_{-\infty}^{\infty} \Phi(u)e\left(\frac{-kXu}{\alpha a^2 p}\right)du \\
     &= \frac{X}{\alpha a^2 p} \sum_{k \in \mathbb{Z}}  
    e\left(\frac{kb}{p}\right)\widehat{\Phi}\left( \frac{kX}{\alpha a^2p}\right).
\end{split}
\end{align}
Multiplying equation~\eqref{eq.PS} by $\left(\frac{b}{p}\right)$, summing over $b$ mod $p$, and switching the order of summation, we get
\begin{align}\label{eq.multsum}
\begin{split}
    \sum_{b \bmod p} \left(\frac{b}{p}\right) \sum_{d \in \mathbb{Z}} \Phi\left(\frac{\alpha a^2 (pd + b)}{X}\right) &= \frac{X}{\alpha a^2 p}  \sum_{k \in \mathbb{Z}} \sum_{b \bmod p} \left(\frac{b}{p}\right)   e\left(\frac{kb}{p}\right)\widehat{\Phi}\left( \frac{kX}{\alpha a^2p}\right) \\
    & =\frac{X}{\alpha a^2 p} \sum_{k \in \mathbb{Z}} \tau_k(p)\widehat{\Phi}\left( \frac{kX}{\alpha a^2p}\right).
\end{split}
\end{align}
Combining equations~\eqref{eq:recombined},~\eqref{eq:odd-split} and~\eqref{eq.multsum}, and using $G_k(p) = \left(\frac{2}{p}\right)G_{2k}(p)$, we deduce
\begin{equation}\label{eq:possion-expansion}
   \sum_{\substack{d \in \mathbb{Z} \\ (d,2)=1}}  \Phi\left(\frac{da^2}{X}\right)\left(\frac{d}{p}\right) = \frac{X}{2a^2p}\left(\frac{2}{p}\right)\sum_{k \in \mathbb{Z}} (-1)^k G_k(p)\widetilde{\Phi}\left(\frac{kX}{2a^2p}\right).
\end{equation}
Since $G_k(p) = \left(\frac{k}{p}\right)\sqrt{p}$, equation~\eqref{eq.SMap} follows from equations~\eqref{eq:sound-identity} and~\eqref{eq:possion-expansion}.
\end{proof}

\begin{lemma}\label{lem:decay}
Let $\Phi$ be a Schwartz function. 
For any $\alpha > 1$, as $X \to \infty$, we have
\begin{equation}
    \sum_{m \in \mathbb{N}} \Phi\left(Xm\right) \ll X^{-\alpha}.
\end{equation}
\end{lemma}
\begin{proof}
Since $\Phi$ is Schwartz, as $X\rightarrow\infty$, we have $\Phi(X) \ll X^{-\alpha}$. We deduce that:
\begin{equation}
    \sum_{m \in \mathbb{N}} \Phi(Xm) \ll \sum_{m \in \mathbb{N}} (Xm)^{-\alpha} = X^{-\alpha} \sum_{m \in \mathbb{N}} m^{-\alpha} \ll X^{-\alpha}.
\end{equation}
\end{proof}

\section{Proof of Theorem~\ref{thm.main}}\label{s:prime}

\begin{proof}[Proof of equation \eqref{eq.Athm-1}]
We will prove the case of $P_+(y,X,c)$, and simply note that the case of $P_- (y, X, c)$ is similar.
For $p\neq N$, equation~\eqref{eq.mcos} implies
\begin{equation}\label{eq.goodstart}
\lim_{X\to \infty} P_+ (y, X, c) = \lim_{X \to \infty} \frac{\log X}{X}\sum_{\substack{N \in [X, cX] \\ N \text{ prime}}} \left [ \left( \frac{N-1}{N}\right)\cos\left(\frac{2 \pi \lceil yX \prceil}{N} \right) + \frac 1 N \right ].  
\end{equation}
With $a=c-1$ and $b=1$ in \eqref{eq.1/N}, we have \begin{equation}\label{eq.1/N-1}   \lim_{X \to \infty} \frac{\log X}{X}\sum_{\substack{N \in [X, cX]}}  \frac 1 N  =0. \end{equation}
Substituting equations \eqref{eq.yx} and \eqref{eq.1/N-1} into equation~\eqref{eq.goodstart} gives:
\begin{equation}\label{eq.goodstart+}
\lim_{X\to \infty} P_+ (y, X, c) =\lim_{X \to \infty} \frac{\log X}{X}\sum_{\substack{N \in [X, cX] \\ N \text{ prime}}}\cos\left(\frac{2 \pi yX}{N} \right).
\end{equation}
We relate the sum on the right hand side of equation~\eqref{eq.goodstart+} to an integral using the following equidistribution argument.
For each $X$, consider the set $S=\{N \in [X, cX] : N \text{ prime}\}$.
If $n=\#S$, then, according to the prime number theorem, we have
\begin{equation}\label{eq.apply-pnt}
n \sim \left ( \frac {cX}{\log(cX)} - \frac X {\log X} \right )\sim \frac {(c-1) X}{\log X}.
\end{equation} 
Consider the sequence $T=(N_i/X)_{i=1}^n$ where $N_i\in S$ for $i\in\{1,\dots,n\}$.
Any subinterval $[\alpha,\beta]\subset(1,c)$ contains the following proportion of elements in $T$:
\[\frac{\pi(\beta X)-\pi(\alpha X)}{n}\sim\frac{\beta X/\log(\beta X)-\alpha X/\log(\alpha X)}{(c-1)X/\log(X)}\sim\frac{\beta-\alpha}{c-1}.\]
In other words, the sequence $T=(N_i/X)_{i=1}^n$ approaches equidistributed on $(1,c)$.
Using equations~\eqref{eq.goodstart+} and~\eqref{eq.apply-pnt}, and applying equidistribution of the sequence $T$ on $(1,c)$, we  conclude using Riemann sums that:
\begin{equation}
\lim_{X\to \infty} P_+ (y, X, c)
    = \lim_{\substack{X \to \infty \\ n \to \infty}} \frac{c-1}{n} \sum_{i=1}^n \cos\left(\frac{2 \pi y X}{N_i} \right) =\int_1^c \cos\left(\frac{2 \pi y}{x}\right)dx.
\end{equation}
\end{proof}

\begin{remark}\label{rem.discont}
The discontinuity around $y=1$ in the bottom image from Figure~\ref{fig:prime_lavg} is explained by the fact that equation~\eqref{eq.goodstart} requires $p\neq N$.
In fact, when $p=N$, the quantity $\chi(\lceil yX \prceil)/\tau(\chi)$ vanishes. 
This discrepancy does not affect the limit.
\end{remark}
\begin{proof}[Proof of equation \eqref{eq.LAthm-1}]
Recall that we assume the Riemann hypothesis.
We will prove the case of $\widetilde P_+(y,X,c)$, and simply note that the case of $\widetilde P_- (y, X, \delta)$ is similar.
Equation~\eqref{eq.mcos} implies that
\begin{equation}\label{eq.thmlem}
 \widetilde P_+ (y, X, \delta)= \frac{\log X}{X^{\delta}} \sum_{\substack{N \in [X, X+X^{\delta}] \\ N \text{ prime}}} \left [ \left(\frac{N-1}{N} \right)\cos\left(\frac{2\pi \lceil y X \prceil}{N}\right) + \frac 1 N \right ].\\
\end{equation}
With $a=1$ and $b=\delta$ in \eqref{eq.1/N}, we obtain
\begin{equation} \label{eq.1/N-2}
    \lim_{X \to \infty} \frac{\log X}{X^\delta} \sum_{\substack{N \in [X, X+X^\delta]}} \frac{1}{N}  = 0.
\end{equation}
Applying equations \eqref{eq.yx} and \eqref{eq.1/N-2} to equation~\eqref{eq.thmlem}, we deduce
\begin{equation}\label{eq.important}
\lim_{X \to \infty} \widetilde P_+ (y, X, \delta) = \lim_{X \to \infty} \frac{\log X}{X^{\delta}} \sum_{\substack{N \in [X, X+X^{\delta}] \\ N \text{ prime}}}\cos\left(\frac{2\pi  y X }{N}\right).
\end{equation}
Now it follows from Lemma \ref{lem:often-used} that  
\begin{equation}
\lim_{X \to \infty} \frac{\log X}{X^{\delta}} \sum_{\substack{N \in [X, X+X^{\delta}] \\ N \text{ prime}}}\cos\left(\frac{2\pi  y X }{N}\right) = \cos(2 \pi y).
\end{equation}
\end{proof}

\section{Proof of Theorem~\ref{thm.main-real}}\label{s:realp}
For $d \in \mathbb Z_{<0}$, we define $\mu(d)=\mu(|d|)$.  
Let $\mathcal{G}$ be as in Section~\ref{s:intro} and note that $d \in \mathcal{G}$ if and only if $(d,2) = 1$ and $\mu^2(d) = 1$. 
Subsequently, for $y$, $\delta$, and $\Phi$ as in Theorem~\ref{thm.main-real}, we may rewrite equation~\eqref{eq.Mp-1} as follows:
\begin{equation}\label{eq.Mp-11}
M_{\Phi}(y,X,\delta)=\frac{\log X}{X^{1+\delta}}\sum_{\substack{p\in[yX,yX+X^{\delta}] \\ p \text{ prime}}}\sum_{\substack{d \in \mathbb{Z} \\ (d,2)=1}}\mu^2(d) \Phi\left(\frac{d}{X}\right)  \chi_{8d}(p)\sqrt{p}.
\end{equation}
According to \cite[equation~(1.33)]{IK}, we have $\mu^2(d) = \sum_{\substack{ a^2 \mid d \\ a >0 } } \mu(a)$, and so:
\begin{equation}\label{eq.ami}
  \sum_{\substack{d \in \mathbb{Z} \\ (d,2)=1}}\mu^2(d) \Phi\left(\frac{d}{X}\right)\chi_{8d}(p)\sqrt{p} =  \sum_{\substack{d \in \mathbb{Z} \\ (d,2)=1}} \Big(\sum_{\substack{ a^2 \mid d \\ a >0 }} \mu(a)\Big)\Phi\left(\frac{d}{X}\right) \chi_{8d}(p)\sqrt{p}.
\end{equation}
Since $\Phi$ has compact support, we may define $\beta = \sup_{x \in \mathbb{R}} \{|x| : \Phi(x) > 0\}<\infty$. 
Combining equation~\eqref{eq.Mp-11} and equation~\eqref{eq.ami}, for $A \in (0, \sqrt{\beta X}]$, we may write $M_{\Phi}(y, X,\delta) = M_{\Phi,A}(y, X, \delta) + R_{\Phi,A}(y,X,\delta)$, where:
\begin{align}
     M_{\Phi,A}(y, X, \delta) &=\frac{\log X}{X^{1+\delta}} \sum_{\substack{p\in[yX,yX+X^{\delta}] \\ p \text{ prime}}} \sum_{\substack{d \in \mathbb{Z} \\ (d,2)=1}} \Big(\sum_{\substack{a^2 \mid d \\ 0< a \leq A}} \mu(a)\Big)\Phi\left(\frac{d}{X}\right)  \chi_{8d}(p)\sqrt{p},\label{eq.MAp}\\
     R_{\Phi,A}(y, X, \delta) &=\frac{\log X}{X^{1+\delta}} \sum_{\substack{p\in[yX,yX+X^{\delta}] \\ p \text{ prime}}} \sum_{\substack{d \in \mathbb{Z} \\ (d,2)=1}}  \Big(\sum_{\substack{a^2 \mid d \\ a > A}} \mu(a)\Big)\Phi\left(\frac{d}{X}\right)  \chi_{8d}(p)\sqrt{p}.\label{eq.RAp}
\end{align}
To complete the proof, we will show that $R_{\Phi, A}(y, X, \delta)$ vanishes as $X \to \infty$, and use \cite[Lemma~2.6]{S00} in the form of Lemma~\ref{lem.Sound-lemma} to analyse the asymptotic behaviour of $M_{\Phi, A}(y, X, \delta)$. 

\subsection{Analysis of $R_{\Phi,A}(y,X, \delta)$}
Given $d\in\mathbb{Z}$, for any $\epsilon>0$, we have
\begin{equation}\label{eq.sm}
    \left\lvert\sum_{\substack{a^2 \mid d \\ a > A}} \mu(a) \right\rvert \ll \sum_{k \mid d} 1 \ll |d|^\epsilon.
\end{equation}
Since the innermost sum in equation~\eqref{eq.RAp} is empty unless $d = a^2b$ where $a > A$, and
$\Phi(d/X) = 0$ unless $|d| < \beta X$, switching the order of summation in equation~\eqref{eq.RAp} and applying equation~\eqref{eq.sm} shows that
\begin{equation}\label{eq.mR}
    \left\lvert R_{\Phi,A}(y, X, \delta)\right\rvert \ll \frac{\log X}{X^{\frac12+\delta-2\epsilon}}  \sum_{a\in (A,\sqrt{\beta X}]} \sum_{|b| \leq \frac{\beta X}{a^2} } \Phi\left(\frac{a^2b}{X}\right)\left|\sum_{\substack{p\in[yX,yX+X^{\delta}] \\ p \text{ prime}}}\chi_{8d}(p)\sqrt{\frac{p}{X}}\right|,
\end{equation}
where the outer sums are over $a,b\in\mathbb{Z}$ satisfying the specified bounds. 
Using Abel's summation formula (\cite[Theorem 4.2]{A76}), we get
\begin{equation} \label{eq.asf}
    \sum_{\substack{p\in[yX,yX+X^{\delta}] \\ p \text{ prime}}} \chi_{8d}(p)\sqrt{\frac{p}{X}}=  \sqrt{\frac{t}{X}}\psi_{8d}(t)\Big\vert_{t = yX}^{yX+X^\delta} - \int_{yX}^{yX+X^\delta} \frac{\psi_{8d}(t)}{2\sqrt{tX}}dt,  
\end{equation}
where we set $\psi_k(t) \seteq \sum_{\substack{3 \leq p \leq t \\ p \text{ prime}}} \left(\frac{k}{p}\right)$. For $t \in [yX,yX+X^\delta]$, we have $|1/\sqrt{t}| \leq (yX)^{-\frac12}$ and, by Lemma~\ref{lem.PV}, $|\psi_k(t)|\ll (yX)^{\frac12 + \epsilon}$. Consequently, we see that
 $|\psi_{8d}(t)/\sqrt{tX}| \ll y^\epsilon X^{\epsilon - \frac 1 2}$. Taking the absolute value of both sides of equation~\eqref{eq.asf}, and recalling $\delta < 1$, we deduce
\begin{equation}\label{eq.pin}
    \left|\sum_{\substack{p\in[yX,yX+X^{\delta}] \\ p \text{ prime}}} \chi_{8d}(p)\sqrt{\frac{p}{X}}\right| \ll y^{1+\epsilon}X^{\frac12 + \epsilon} + y^\epsilon X^{\delta - \frac12 + \epsilon} \ll y^{1+\epsilon}X^{\frac12 + \epsilon}. 
\end{equation}
Applying equation~\eqref{eq.pin} to equation~\eqref{eq.mR}, we infer that:
\begin{equation}\label{eq.Rp}
    \left\lvert R_{\Phi,A}(y, X, \delta)\right\rvert \ll \frac{y^{1+\epsilon}\log X}{X^{\delta - 3\epsilon}}   \sum_{a\in(A,\sqrt{\beta X}]} \sum_{|b| \leq \frac{\beta X}{a^2} } \Phi\left(\frac{a^2b}{X}\right).
\end{equation}
Since $\Phi$ is Schwartz, we have $\Phi(a^2b/X) \leq \sup_{x \in \mathbb{R}}\Phi(x) < \infty$, and so 
\begin{equation}\label{eq.ll}
 \sum_{ a\in(A,\sqrt{\beta X}]} \sum_{|b| \leq \frac{\beta X}{a^2} }\Phi\left(\frac{a^2b}{X}\right) 
    \ll  \sum_{ a\in (A,\sqrt{\beta X}]} \sum_{|b| \leq \frac{\beta X}{a^2} }1 
    \ll X\sum_{ a\in (A,\sqrt{\beta X}]}  \frac{1}{a^2}
    \leq X\int_{A}^\infty \frac{da}{a^2}
    =\frac{X}{A}.
\end{equation}
Combining equations~\eqref{eq.Rp} and~\eqref{eq.ll}, we conclude that
\begin{equation} \label{eq.ll-1}
    |R_{\Phi,A}(y,X,\delta)| \ll \frac{ y^{1+\epsilon} X^{1+3\epsilon - \delta}}{A}\log X \ll \frac{ y^{1+\epsilon} X^{1+4\epsilon - \delta}}{A}.
\end{equation}
In what follows, we will refine our choice of $\epsilon$ and $A$. 
These refinements are made not only to show that $R_{\Phi,A}(y,X,\delta)$ vanishes in the limit, but moreover to find an asymptotic formula for $M_{\Phi,A}(y,X,\delta)$ in the sequel.
Recall from the discussion above equation~\eqref{eq.MAp} that, by construction, we have $A\ll X^{\frac12}$.
For the asymptotic formula, we will require the stronger assumption that $A\ll X^{\frac14}$. 
Since $\delta > 3/4$ is fixed, we may choose $0<\epsilon < (\delta - 3/4 )/5$ and $A = X^{1+5\epsilon-\delta} \ll X^{\frac14} \ll X^{\frac12}$. 
With these choices, equation~\eqref{eq.ll-1} implies that:
\begin{equation}\label{eq.RllXe}
    |R_{\Phi,A}(y,X,\delta)| \ll y^{1+\epsilon}X^{-\epsilon}.
\end{equation}
Using equation~\eqref{eq.RllXe}, and the fact that $M_{\Phi}(y, X,\delta) = M_{\Phi,A}(y, X, \delta) + R_{\Phi,A}(y,X,\delta)$, we obtain 
\begin{equation}\label{eq:M_phi=M_phiA}   \lim_{X\rightarrow\infty}M_{\Phi}(y,X,\delta)=\lim_{X\rightarrow\infty}M_{\Phi,A}(y,X,\delta).
\end{equation}
\subsection{Analysis of $M_{\Phi,A}(y,X, \delta)$}\label{eq.anoMPA}
Recalling that $\chi_{8d}(p)=\left(\frac{8d}{p}\right)$ and applying Lemma~\ref{lem.Sound-lemma}, we deduce 
\begin{equation}
    M_{\Phi,A}(y,X,\delta) = \frac{\log X}{X^\delta} \sum_{\substack{p\in[yX,yX+X^{\delta}] \\ p \text{ prime}}} \frac{1}{2} \left(\frac{16}{p}\right)\sum_{\substack{(a, 2p) = 1\\ 0<a \leq A}} \frac{\mu(a)}{a^2} \sum_{k \in \mathbb{Z}}  (-1)^k \left( \frac{k}{p}\right)\widetilde{\Phi}\left(\frac{kX}{2a^2 p} \right).\label{eq.MpAldsp}
\end{equation}
Since the $k = 0$ term in equation~\eqref{eq.MpAldsp} is identically zero, and $\left(\frac{16}{p}\right) = \left(\frac{4}{p}\right)^2 = 1$ for odd primes $p$, we have
\begin{equation} \label{eq:MpAsimp}
    M_{\Phi,A}(y,X,\delta) = \frac{\log X}{X^\delta} \sum_{\substack{p\in[yX,yX+X^{\delta}] \\ p \text{ prime}}} \frac{1}{2} \sum_{\substack{(a, 2p) = 1\\0<a \leq A}} \frac{\mu(a)}{a^2} \sum_{\substack{k \in \mathbb{Z} \\ k \neq 0}}  (-1)^k \left( \frac{k}{p}\right)\widetilde{\Phi}\left(\frac{kX}{2a^2 p} \right),
\end{equation}
for $X$ sufficiently large (so that $yX >2$). 
Since $0<a \leq A \ll X^{\frac12} \ll p$, we have $(a, 2p) = 1$ if and only if $(a,2) = 1$. Therefore, for large $X$, we can switch the order of summation in equation~\eqref{eq:MpAsimp} to get
\begin{equation}\label{eq.MpAyXd12}
M_{\Phi,A}(y,X,\delta) = \frac{1}{2} \sum_{\substack{(a, 2) = 1\\0<a \leq A}} \frac{\mu(a)}{a^2} \sum_{\substack{k\in \mathbb{Z} \\ k \neq 0}} (-1)^k  \frac{\log X}{X^\delta} \sum_{\substack{p\in[yX,yX+X^{\delta}] \\ p \text{ prime}}}  \left( \frac{k}{p}\right)\widetilde{\Phi}\left(\frac{kX}{2a^2 p} \right).
\end{equation}
We handle the sum over non-zero integers $k$ in equation~\eqref{eq.MpAyXd12} in two stages. 
In the first stage, we break it into sums over $k$ square (written $k=\square$) and $k$ non-square (written $k\neq\square$). 
 In the second stage, we show that the sum over $k\neq\square$ exhibits cancellation, and consequently we identify the sum over $k=\square$ as the main contribution. 
To bound the sum over $k\neq\square$, we introduce $\epsilon'>0$ and break the sum further into sums over $k$ small ($|k| < a^2 X^{\epsilon'}$) and $k$ big ($|k| \geq a^2 X^{\epsilon'}$). The small $k$ are handled by equation~\eqref{eq.MAPV}, which will be deduced in the next paragraph, and the big $k$ are handled by the rapid decay of $\widetilde{\Phi}$.

Assume $k \neq \square$, so that $\left(\frac{k}{\cdot}\right)$ is a non-principal character. Using Abel's summation formula, we have
\begin{equation}\label{eq.MAAS}
\sum_{\substack{p\in[yX,yX+X^{\delta}] \\ p \text{ prime}}}  \left( \frac{k}{p}\right)\widetilde{\Phi}\left(\frac{kX}{2a^2 p} \right) = \widetilde{\Phi}\left(\frac{kX}{2a^2t} \right)\psi_k(t)\Big\vert_{t=yX}^{yX+X^\delta} - \int_{yX}^{yX+X^\delta} \frac{d}{dt}\left(\widetilde{\Phi}\left(\frac{kX}{2a^2 t}\right)\right)\psi_k(t)dt,
\end{equation}
where we set $\psi_k(t) = \sum_{\substack{3 \leq p \leq t \\ p \text{ prime}}} \left(\frac{k}{p}\right)$ as before. 
Since $\widetilde{\Phi}$ is bounded, applying Lemma~\ref{lem.PV} to equation~\eqref{eq.MAAS}, for $\epsilon'>0$ and $a\geq1$, we deduce
\begin{equation}\label{eq.MAPV}
\sum_{\substack{p\in[yX,yX+X^{\delta}] \\ p \text{ prime}}}  \left( \frac{k}{p}\right)\widetilde{\Phi}\left(\frac{kX}{2a^2 p} \right) \ll (yX)^{\frac12+\epsilon'}\left(1 + \int_0^\infty \left\lvert \widetilde{\Phi}'(u)  \right\rvert du\right) \ \ (k\neq\square).
\end{equation}
Now, summing over $k \neq \square$ with $|k| < a^2 X^{\epsilon'}$, we observe:
\begin{equation}\label{eq.bound_nosq_smallk}
   \sum_{\substack{k \in \mathbb{Z} \\ |k| < a^2X^{\epsilon'}\\ k \neq \square}}    (-1)^k\frac{\log X}{X^\delta} \sum_{\substack{p\in[yX,yX+X^{\delta}] \\ p \text{ prime}}}  \left( \frac{k}{p}\right)\widetilde{\Phi}\left(\frac{kX}{2a^2 p} \right) \ll \frac{y^{\frac12 + \epsilon'} a^2 \log X}{X^{\delta - \frac12 -2\epsilon'}}.
\end{equation}
On the other hand, summing over $k \neq \square$ with $|k| \geq a^2 X^{\epsilon'}$ and $a\geq1$, we obtain 
\begin{equation}\label{eq.largek_new1}
\begin{split}
    \left|\sum_{\substack{k \in \mathbb{Z} \\ |k| \geq a^2X^{\epsilon'}\\ k \neq \square}}   (-1)^k\frac{\log X}{X^\delta} \sum_{\substack{p\in[yX,yX+X^{\delta}] \\ p \text{ prime}}}  \left( \frac{k}{p}\right)\widetilde{\Phi}\left(\frac{kX}{2a^2 p} \right)\right|
    \ll&\sum_{\substack{k \in \mathbb{Z} \\ |k| \geq a^2X^{\epsilon'}\\ k \neq \square}}\left\lvert\widetilde{\Phi}\left(\frac{k}{2a^2 y} \right)\right\lvert,
\end{split}
\end{equation}
where we use $\left|(-1)^k\left( \frac{k}{p}\right)\right| \leq 1$ and apply Lemma \ref{lem:often-used}. 
Since $\widetilde{\Phi}$ is Schwartz, and $x \mapsto |x|^{-\alpha}$ is even for all $\alpha > 1$, we have 
\begin{equation}\label{eq.largek_new2}
    \sum_{\substack{k \in \mathbb{Z} \\ |k| \geq a^2X^{\epsilon'}\\ k \neq \square}}\left\lvert\widetilde{\Phi}\left(\frac{k}{2a^2 y} \right)\right\lvert \ll 2\int_{a^2X^{\epsilon'}-1}^{\infty} \left(\frac{u}{2a^2y}\right)^{-\alpha}du \ll a^2y^{\alpha} X^{\epsilon'(1-\alpha)}.
\end{equation}
Combining equation ~\eqref{eq.bound_nosq_smallk} with equations~\eqref{eq.largek_new1} and ~\eqref{eq.largek_new2}, we deduce: 
\begin{align}\label{eq.skN2ks}
\begin{split}
\left\lvert\sum_{\substack{k \in \mathbb{Z}\\ k \neq \square}}   (-1)^k\frac{\log X}{X^\delta} \sum_{\substack{p\in[yX,yX+X^{\delta}] \\ p \text{ prime}}}  \left( \frac{k}{p}\right)\widetilde{\Phi}\left(\frac{kX}{2a^2 p} \right)\right\rvert \ll \frac{a^2y^{\frac12 + \epsilon'}\log X}{X^{\delta - \frac12 - 2\epsilon'}} + a^2y^{\alpha}X^{\epsilon'(1-\alpha)}.
\end{split}
\end{align} 
Summing equation~\eqref{eq.skN2ks} over odd $a\leq A$, we see that:
\begin{equation}\label{eq.skN2ks2}
    \left\lvert\sum_{\substack{(a, 2) = 1\\0<a \leq A}} \frac{\mu(a)}{a^2}\sum_{\substack{k \in \mathbb{Z}\\ k \neq \square}}   (-1)^k\frac{\log X}{X^\delta} \sum_{\substack{p\in[yX,yX+X^{\delta}] \\ p \text{ prime}}}  \left( \frac{k}{p}\right)\widetilde{\Phi}\left(\frac{kX}{2a^2 p} \right)\right\rvert \ll A\left(\frac{y^{\frac12 + \epsilon'}\log X}{X^{\delta - \frac12 - 2\epsilon'}} + y^{\alpha}X^{\epsilon'(1-\alpha)}\right).
\end{equation}
Recall from the discussion preceding equation~\eqref{eq.RllXe} that we have chosen 
$A \ll X^{\frac14}$.   
Combining this bound for $A$ with equations~\eqref{eq:MpAsimp} and \eqref{eq.skN2ks2}, we conclude:
\begin{align}\label{eq:M_A_squres}
\begin{split}
    M_{\Phi,A}(y,X,\delta) =&  \frac{\log X}{X^\delta}\sum_{\substack{p\in[yX,yX+X^{\delta}] \\ p \text{ prime}}}\frac{1}{2}\sum_{\substack{0<a \leq A \\ (a, 2) = 1}} \frac{\mu(a)}{a^2} \sum_{\substack{k\in \mathbb{Z} \\ k = \square}} (-1)^k    \left( \frac{k}{p}\right)\widetilde{\Phi}\left(\frac{kX}{2a^2 p} \right) \\
    &+ O\left(\frac{y^{\frac12+\epsilon'}\log X}{X^{\delta - \frac34 -2\epsilon'}} + y^{\alpha}X^{\frac{1}{4}+\epsilon'(1-\alpha)}\right).
\end{split}
\end{align}
Since
$\delta > 3/4$,
we may choose 
$0<\epsilon'<\delta/2-3/8$ and $\alpha > 1 + (4\epsilon')^{-1}>1$. 
With these choices, we see that the error term in equation~\eqref{eq:M_A_squres} vanishes in the limit as $X\rightarrow\infty$, and hence:
\begin{equation}\label{eq:M_A_squres-e}
    \lim_{X\rightarrow\infty}M_{\Phi,A}(y,X,\delta) =  \lim_{X\rightarrow\infty}\frac{\log X}{X^\delta}\sum_{\substack{p\in[yX,yX+X^{\delta}] \\ p \text{ prime}}}\frac{1}{2}\sum_{\substack{0<a \leq A \\ (a, 2) = 1}} \frac{\mu(a)}{a^2} \sum_{\substack{k\in \mathbb{Z} \\ k = \square}} (-1)^k    \left( \frac{k}{p}\right)\widetilde{\Phi}\left(\frac{kX}{2a^2 p} \right).
\end{equation}
Considering the innermost sum in equation~\eqref{eq:M_A_squres-e}, we note that:
\begin{equation}\label{eq.innermost}
\sum_{\substack{k \in \mathbb{Z}\\ k = \square}}  (-1)^k \left( \frac{k}{p}\right)\widetilde{\Phi}\left(\frac{kX}{2a^2 p} \right) 
= \sum_{m=1}^\infty  (-1)^m \left( \frac{m^2}{p}\right)\widetilde{\Phi}\left(\frac{m^2X}{2a^2 p} \right) 
= \sum_{\substack{m=1 \\ (m, p)=1}}^\infty  (-1)^m \widetilde{\Phi}\left(\frac{m^2X}{2a^2 p} \right).
\end{equation}
Combining equations~\eqref{eq:M_A_squres-e} and~\eqref{eq.innermost}, we deduce
\begin{equation}\label{eq:M_A_squres-e+m}
    \lim_{X\rightarrow\infty}M_{\Phi,A}(y,X,\delta) =  \lim_{X\rightarrow\infty}\frac{\log X}{X^\delta}\sum_{\substack{p\in[yX,yX+X^{\delta}] \\ p \text{ prime}}}\frac{1}{2}\sum_{\substack{0<a \leq A \\ (a, 2) = 1}} \frac{\mu(a)}{a^2} \sum_{\substack{m=1 \\ (m, p)=1}}^\infty  (-1)^m \widetilde{\Phi}\left(\frac{m^2X}{2a^2 p} \right).
\end{equation}
To analyse the sum over $m$ coprime to $p$ in~\eqref{eq:M_A_squres-e+m}, we will eventually apply Poisson summation.
Prior to that, we will first quantify the error created when we extend the domain of summation to all $m>0$.
Since $\{m \in \mathbb{N} : (p, m) > 1\} = \{pm : m \in \mathbb{N}\}$ and $\widetilde{\Phi}$ is Schwartz, Lemma $\ref{lem:decay}$ implies that, for all $\kappa >1$, we have:
\begin{equation}\label{eq.bound(p,m)}
\sum_{\substack{m=1 \\ (p, m) > 1}}^\infty  (-1)^m \widetilde{\Phi}\left(\frac{m^2X}{2a^2 p} \right) \leq \sum_{\substack{m=1}}^\infty \left|\widetilde{\Phi}\left(\frac{pm^2X}{2a^2} \right) \right| \ll \left(\frac{pX}{2a^2}\right)^{-\kappa}.
\end{equation}
Since $a \leq A \ll \sqrt{X} \sim \sqrt{p/y}$, we have $pX/a^2 \gg yX$.  
Combining these bounds with equation~\eqref{eq.bound(p,m)}, we deduce
\begin{equation}\label{eq.alpapp}
    \sum_{\substack{0<a \leq A \\ (a, 2) = 1}} \frac{\mu(a)}{a^2}\sum_{\substack{m=1 \\ (p, m) > 1}}^\infty  (-1)^m \widetilde{\Phi}\left(\frac{m^2X}{2a^2 p} \right) \ll  (yX)^{-\kappa}\sum_{\substack{0<a \leq A \\ (a, 2) = 1}} \frac{1}{a^2} \ll (yX)^{-\kappa}.
\end{equation} 
Equation~\eqref{eq.alpapp} implies that
\begin{equation}\label{eq.main_term_full}
     \sum_{\substack{0<a \leq A \\ (a, 2) = 1}} \frac{\mu(a)}{a^2}\sum_{\substack{m=1 \\ (m, p)=1}}^\infty  (-1)^m \widetilde{\Phi}\left(\frac{m^2X}{2a^2 p} \right) = \sum_{\substack{0<a \leq A \\ (a, 2) = 1}} \frac{\mu(a)}{a^2}\sum_{\substack{m=1}}^\infty  (-1)^m \widetilde{\Phi}\left(\frac{m^2X}{2a^2 p} \right) + O\left((yX)^{-\kappa}\right).
\end{equation}
Applying equation~\eqref{eq.main_term_full} to equation~\eqref{eq.innermost} and combining the result with \eqref{eq:M_A_squres-e}, we deduce:
\begin{equation}\label{eq:M_A_squres-e+m1}
    \lim_{X\rightarrow\infty}M_{\Phi,A}(y,X,\delta) =  \lim_{X\rightarrow\infty}\frac{\log X}{X^\delta}\sum_{\substack{p\in[yX,yX+X^{\delta}] \\ p \text{ prime}}}\frac{1}{2}\sum_{\substack{0<a \leq A \\ (a, 2) = 1}} \frac{\mu(a)}{a^2} \sum_{m=1}^\infty  (-1)^m \widetilde{\Phi}\left(\frac{m^2X}{2a^2 p} \right).
\end{equation}
To analyse the inner sum on the right hand side of equation~\eqref{eq:M_A_squres-e+m1}, we observe that
\begin{equation}\label{eq.poission_form}
    \sum_{\substack{m=1}}^\infty  (-1)^m \widetilde{\Phi}\left(\frac{m^2X}{2a^2 p} \right) = \frac{1}{2}\sum_{m \in \mathbb{Z}}(-1)^m \widetilde{\Phi}\left(\frac{m^2X}{2a^2 p} \right) - \frac{1}{2}\widetilde{\Phi}(0).
\end{equation}
Poisson summation implies that
\begin{align}
\begin{split}\label{eq.possion_2}
    \sum_{\substack{m \in \mathbb{Z}}}  (-1)^m \widetilde{\Phi}\left(\frac{m^2X}{2a^2 p} \right) &= \sum_{m \in \mathbb{Z}}\cos\left(\pi m\right) \widetilde{\Phi}\left(\frac{m^2 X}{2a^2p}\right) \\
   &= \sum_{v \in \mathbb{Z}} \int_{-\infty}^\infty \widetilde{\Phi}\left(\frac{u^2 X}{2a^2 p}\right)\cos\left(\pi u\right) 
    e(-u v)du \\
    &= a\sqrt{\frac{2p}{X}}\sum_{v \in \mathbb{Z}} \int_{-\infty}^\infty \widetilde{\Phi}\left(w^2\right)\cos\left(\pi wa \sqrt{\frac{2p}{X}}\right) 
    e\left(-wav\sqrt{\frac{2p}{X}}\right)dw \\
    &=  a\sqrt{\frac{2p}{X}}\left(\widehat{H}_a(0) + 2\sum_{v =1}^\infty \widehat{H}_a\left(av\sqrt{\frac{2p}{X}}\right)\right),
\end{split}    
\end{align}
where $H_a(w) \seteq \tilde{\Phi}(w^2)\cos(\pi  w a\sqrt{2p/X})$. For the final equality in equation~\eqref{eq.possion_2}, we use the fact that $\widehat{H}$ is even.
To proceed, let $H(w) \seteq  \tilde{\Phi}(w^2)$ and $C_a(w) \seteq \cos(\pi  w a\sqrt{2p/X})$ so $H_a(w) = H(w)C_a(w)$. 
Since the Fourier transform of a product is the convolution of the Fourier transforms, we have:  
\begin{align}\label{eq.fourier-product}
\begin{split}
     \widehat{H}_a(w) &= (\widehat{H} \star \widehat{C_a})(w) \\&= 
    \frac{1}{2}\int_{-\infty}^{\infty} \widehat{H}(t)\left(\delta\left(w-t-a\sqrt{\frac{p}{2X}}\right) + \delta\left(w-t+a\sqrt{\frac{p}{2X}}\right)\right)dt \\ &=
    \frac12\left(\widehat{H}\left(w + a\sqrt{\frac{p}{2X}}\right) + \widehat{H}\left(w-a\sqrt{\frac{p}{2X}}\right)\right).  
\end{split}
\end{align}
Since $\widehat{H}$ is even, equation~\eqref{eq.fourier-product} implies:
\begin{align}\label{eq.fourier-sum}
\begin{split}
    \widehat{H}_a(0) + 2\sum_{v =1}^\infty \widehat{H}_a\left(av\sqrt{\frac{2p}{X}}\right) &= \widehat{H}\left(a\sqrt{\frac{p}{2X}} \right) + \sum_{v=1}^\infty \widehat{H}\left(a\sqrt{\frac{p}{2X}}(2v+1)\right) + \widehat{H}\left(a\sqrt{\frac{p}{2X}}(2v-1)\right) \\
    &= 2\sum_{\substack{v=1 \\ (v,2)=1}}^\infty \widehat{H}\left(av\sqrt{\frac{p}{2X}}\right).
\end{split}    
\end{align}
Substituting \eqref{eq.fourier-sum} into equation~\eqref{eq.possion_2}, we get:
\begin{align}\label{eq.possion_4}
\begin{split}
       \sum_{\substack{m \in \mathbb{Z}}}  (-1)^m \widetilde{\Phi}\left(\frac{m^2X}{2a^2 p} \right) 
       &=  2a\sqrt{\frac{2p}{X}} \sum_{\substack{v=1 \\ (v,2)=1}}^\infty \widehat{H}\left(av \sqrt{\frac{p}{2X}}\right).
\end{split}
\end{align}
Combining equations~\eqref{eq.poission_form} and~\eqref{eq.possion_4}, we deduce:
\begin{align}\label{eq.possion_3}
\begin{split}
    \sum_{m=1}^\infty (-1)^m \widetilde{\Phi}\left(\frac{m^2 X}{2a^2p}\right) &=  - \frac{1}{2}\widetilde{\Phi}(0) + a\sqrt{\frac{2p}{X}}  \sum_{\substack{v=1 \\ (v,2)=1}}^\infty \widehat{H}\left( av\sqrt{\frac{p}{2X}}\right).
\end{split}
\end{align}
Noting that 
\begin{equation}
    \sum_{\substack{0<a \leq A \\ (a,2) = 1}} \frac{\mu(a)}{a^2} = \prod_{p > 2} \left(1 - \frac{1}{p^2}\right) + o(1) = \frac{8}{\pi^2} + o(1), 
\end{equation}
as  $X\rightarrow\infty$ (hence $A=X^{1+5\epsilon-\delta} \to \infty$), equation~\eqref{eq.possion_3} implies:
\begin{equation}\label{eq.poission_result}
\begin{split}
    \sum_{\substack{0<a \leq A \\ (a, 2) = 1}} \frac{\mu(a)}{a^2}\sum_{\substack{m=1}}^\infty  (-1)^m \widetilde{\Phi}\left(\frac{m^2X}{2a^2 p} \right) =  -\frac{4}{\pi^2}\widetilde{\Phi}(0)+\sqrt{\frac{2p}{X}}\sum_{\substack{0<a \leq A \\ (a, 2) = 1}} \frac{ \mu(a)}{a} \sum_{\substack{v=1 \\ (v,2)=1}}^\infty \widehat{H}\left( av \sqrt{\frac{p}{2X}}\right) + o(1).
\end{split}
\end{equation}
Since $\widehat{H}$ is Schwartz,  the double sum in the right hand side of equation~\eqref{eq.poission_result} converges, and we conclude that the sum on the left hand side converges to a smooth function of $p/X\sim y$ as $X\rightarrow\infty$ (hence $A=X^{1+5\epsilon-\delta} \to \infty$).  
Noting the appearance of this sum in equation~\eqref{eq:M_A_squres-e+m1}, we may apply Lemma~\ref{lem:often-used} to deduce: 
\begin{equation}\label{eq.almost}
\begin{split}
\lim_{X \to \infty}  M_{\Phi,A}(y,X,\delta) =  \frac{1}{2}\sum_{\substack{a =1 \\ (a,2)=1}}^\infty \frac{\mu(a)}{a^2} \sum_{m =1}^\infty (-1)^m \widetilde{\Phi}\left(\frac{m^2}{2a^2 y}\right).
\end{split}
\end{equation}
Combining equations~\eqref{eq:M_phi=M_phiA} and~\eqref{eq.almost}, we conclude the proof of Theorem \ref{thm.main-real}.

\section{$1$-level density}\label{s:1ld}
In \cite[Section~4]{C23}, murmurations of Kronecker symbols are considered from the perspective of $L$-function zeros via the explicit formula.
Furthermore, Rubinstein--Sarnak observed that the murmuration density for Kronecker symbols in Theorem~\ref{thm.main-real}, when properly normalized, interpolates the transition in the 1-level densities for a symplectic family of $L$-functions \cite{S23ii}.
We recover the observation of Rubinstein--Sarnak in Corollary~\ref{c:1ld}, in which the left (resp. right) limit corresponds to the case that $p=yX$ is much smaller (resp. larger) than $X$. 

\begin{proof}[Proof of Corollary~\ref{c:1ld}]
We maintain the notation from Corollary~\ref{c:1ld}.
Combining equations~\eqref{eq.limreal} and~\eqref{eq.poission_result} with $y = p/X$, we observe that
\begin{align}\label{eq:m_phi_divsum}
\begin{split}
    M_\Phi(y, \delta) &= -\frac{2}{\pi^2}\widetilde\Phi(0) + \sqrt{\frac{y}{2}}\sum_{\substack{a=1 \\ (a,2)=1}}^\infty \frac{\mu(a)}{a} \sum_{\substack{v=1 \\ (v,2)=1}}^\infty \widehat{H}\left( av\sqrt{\frac{y}{2}}\right)\\
    &= -\frac{2}{\pi^2}\widetilde\Phi(0) + \sqrt{\frac{y}{2}} \sum_{\substack{n=1 \\ (n,2)=1}}^\infty \left(\sum_{\substack{a \mid n \\ (a,2)=1}} \frac{\mu(a)}{a} \right)\widehat{H}\left( n\sqrt{\frac{y}{2}}\right).
 \end{split}   
 \end{align}
For even $n$, define $b_n = 0$, and, for odd $n$, define
\begin{equation}\label{eq.defbn}
    b_n = \sum_{\substack{a \mid n \\ (a,2)=1}} \frac{\mu(a)}{a}.
\end{equation}
In particular, we have $b_1=1$.
Since, for $p$ prime and $m > 1$, we have $\mu(p^m) = 0$, we observe that, for $k \ge 1$,
 \begin{equation}\label{eq.bpk}
     b_{p^k} = \begin{cases}
         0, & p = 2,\\
         1- \frac{1}{p}, & p > 2.
     \end{cases}
 \end{equation}
If $B(s) = \sum_{n=1}^\infty b_nn^{-s}$, then equation~\eqref{eq.bpk} implies that, for $\mathrm{Re}(s)>1$,  
 \begin{align}
 \begin{split}\label{eq.Bszeta}
      B(s) &= \prod_{p > 2} \left( 1 + \left(1 - \frac{1}{p}\right)\sum_{k=1}^\infty p^{-ks} \right) =\prod_{p > 2}\left(1 + \frac{(1-1/p)p^{-s}}{1-p^{-s}}\right) 
      =  \prod_{p > 2} \frac{1 - p^{-s-1}}{1-p^{-s}}\\& = \frac {1-2^{-s}}{1-2^{-s-1}} \ \frac{\zeta(s)}{\zeta(s+1)} .
 \end{split}
 \end{align}
Equation~\eqref{eq.Bszeta} implies that $B(s)$  has meromorphic continuation to $\mathrm{Re}(s) > 0$ with a simple pole at $s=1$ and 
 \begin{equation}\label{eq.Rs1B}
     \mathrm{Res}_{s=1}B(s) = \lim_{s \to 1} (s-1) \frac {1-2^{-s}}{1-2^{-s-1}} \ \frac{\zeta(s)}{\zeta(s+1)} = \frac 2 3 \frac 1 {\zeta(2)} = \frac{4}{\pi^2}.
 \end{equation}
Consequently, we observe that the following function is meromorphic on $\mathrm{Re}(s) > 0$  with a simple pole at $s=1$:
 \begin{equation}\label{eq.MsbH}
     \int_0^\infty \left(\sum_{n=1}^\infty b_n\widehat{H}(nx)\right)x^s \frac{dx}{x} = \int_0^\infty \left(\sum_{n=1}^\infty b_n n^{-s} \widehat{H}(u) \right) u^s \frac{du}{u} = B(s) \int_0^\infty \widehat{H}(u) u^s \frac{du}{u}.
 \end{equation}
 Furthermore, equation~\eqref{eq.Rs1B} implies that:
 \begin{equation}\label{eq.Rse1}
     \mathrm{Res}_{s=1}
     \left(B(s)\int_0^\infty \widehat{H}(u) u^s \frac{du}{u} \right)= \frac{4}{\pi^2} \int_0^\infty \widehat{H}(u)du = \frac{2}{\pi^2} H(0) = \frac{2}{\pi^2} \widetilde{\Phi}(0),
 \end{equation}
 where we have used Fourier inversion, and the facts that $\hat{H}$ is even and $H(0) =\widetilde{\Phi}(0)$. We observe that, for $\eta \in (0, \frac12)$ and $1-\eta \leq \mathrm{Re}(s) \leq 1+ \eta$, the Riemann hypothesis implies that, for all $\epsilon > 0$, we have
 \begin{equation}
     \left| B(s) \right| = \left| \frac {1-2^{-s}}{1-2^{-s-1}}\frac{\zeta(s)}{\zeta(s+1)} \right| \ll \left| \frac{\zeta(s)}{\zeta(s+1)} \right| \ll |s|^\epsilon, 
 \end{equation}
as $|s| \to \infty$ (see \cite[Theorem 14.2]{T87}). Likewise, for $1-\eta \leq \mathrm{Re}(s) \leq 1+ \eta$ and $r \in \mathbb{Z}_{\geq 1}$, we may apply applying integration by parts $r$ times, and note that $\widehat{H}^{(r)}$ is a bounded function of rapid decay, to deduce: 
\begin{equation}\label{eq.i0ihh}
    \left|\int_0^\infty \widehat{H}(u)u^s \frac{du}{u} \right| = \left|\frac{(-1)^r(s-1)!}{(s+r-1)!} \int_0^\infty \widehat{H}^{(r)}(u)u^{s+r}\frac{du}{u}\right| \leq |s|^{-r} \int_0^\infty \left|\widehat{H}^{(r)}(u)\right|u^{r+\eta} du   \ll |s|^{-r},
\end{equation}
as $|s| \to \infty$. Therefore we may apply Mellin inversion, the residue theorem, and equation~\eqref{eq.Rse1} to equation~\eqref{eq.MsbH} 
 to obtain
\begin{equation}\label{eq.sn1ibhh}
    \sum_{n=1}^\infty b_n \widehat{H}(xn) \sim \text{Res}_{s=1}\left(B(s) \int_0^\infty \widehat{H}(u)u^s \frac{du}{u} x^{-s}\right) = \frac{2\widetilde{\Phi}(0)}{\pi^2 x}
\end{equation}
as $x \to 0^+$.
Combining equations~\eqref{eq.defbn} and~\eqref{eq.sn1ibhh}, we deduce that
 \begin{equation}\label{eq.2op2tp0}
    \sqrt{\frac{y}{2}} \sum_{\substack{n=1 \\ (n,2)=1} }^\infty \left(\sum_{\substack{a \mid n \\ (a,2)=1}} \frac{\mu(a)}{a} \right)\widehat{H}\left( n\sqrt{\frac{y}{2}}\right) \sim \sqrt{\frac{y}{2}}\frac{2\widetilde{\Phi}(0)}{\pi^2 \sqrt{y/2}} = \frac{2}{\pi^2} \widetilde{\Phi}(0)
 \end{equation}
 as $y \to 0^+$. Combining equation~\eqref{eq:m_phi_divsum} with equation~\eqref{eq.2op2tp0}, we conclude that
 \begin{equation}
     \lim_{y \to 0^+} M_\Phi(y,\delta) = -\frac{2}{\pi^2}\widetilde{\Phi}(0) + \frac{2}{\pi^2}\widetilde{\Phi}(0) = 0.
 \end{equation}
For the limit as $y\rightarrow\infty$, we again use equation~\eqref{eq.poission_result} with $p/X = y$, and note that, since $\widehat{H}$ is Schwartz, for all $\alpha > 1$, Lemma \ref{lem:decay} implies that 
\begin{equation}\label{eq.fml}
    \left|\sqrt{2y}\sum_{\substack{a =1 \\ (a, 2) = 1}}^\infty \frac{\mu(a)}{a} \sum_{v =1}^\infty \widehat{H}\left(\sqrt{\frac{y}{2}}av\right)\right| \ll  \sum_{\substack{a =1 \\ (a, 2) = 1}}^\infty \frac{\sqrt{y}}{a} \left(\sqrt{\frac{y}{2}}a\right)^{-\alpha} \ll y^{\frac{1-\alpha}{2}} \sum_{\substack{a=1 \\ (a,2)=1}}^{\infty} \frac{1}{a^{\alpha + 1}} \ll y^{\frac{1-\alpha}{2}}.
\end{equation}
Combining equations~\eqref{eq:m_phi_divsum} and~\eqref{eq.fml} we conclude that 
\begin{equation}
    \lim_{y\rightarrow \infty} M_{\Phi}(y, \delta) = \lim_{y\rightarrow \infty}  \frac 1 2  \left(\frac{-4 \widetilde{\Phi}(0)}{\pi^2} + O\left(y^{\frac{1-\alpha}{2}}\right)\right) = -\frac{2}{\pi^2} \widetilde \Phi(0).
\end{equation}
\end{proof}

\section{Supplementary results}\label{s:generalconductors}
\subsection{Including composite conductors in Theorem~\ref{thm.main} }
\subsubsection{Preliminaries}
Note that by \cite[equation~(3.7)]{IK}, the set $\mathcal{D}_\pm(N)$ is empty if and only if $N \equiv 2 \bmod 4$. 
For $\delta\in(0,1)$, $y\in\mathbb{R}_{>0}$, and $c \in \mathbb{R}_{>1}$, we will analyse functions connected to the following:\begin{align}
Q_{\pm}(y,X, c)=&\frac{1}{X}\sum_{\substack{N \in[X, cX] \\ N \not\equiv 2 \bmod 4}} \sum_{\chi \in \mathcal{D}_\pm(N)} \frac{\chi(\lceil yX \prceil)}{\tau(\chi)},\label{eq.Mpm}\\
\widetilde{Q}_{\pm}(y,X,\delta)=&\frac{1}{X^\delta}\sum_{\substack{N \in[X, X+X^\delta] \\ N \not\equiv 2 \bmod 4}} \sum_{\chi \in \mathcal{D}_\pm(N)} \frac{\chi(\lceil yX \prceil)}{\tau(\chi)}. 
\end{align}
For an integer $N>1$ and a prime number $p$ coprime to $N$, Lemma~\ref{lem.cosi} implies:
\begin{align}
\sum_{\substack{\chi \bmod N \\ \ \chi \neq \chi_0, \, \chi(-1) = 1}} \tau(\overline{\chi}) \chi(p)&=1+\phi(N)\cos \left(\frac{2 \pi p}N \right), \label{eq.cos-2}\\
\sum_{\substack{\chi \bmod N \\ \chi(-1) = -1}} \tau(\overline{\chi}) \chi(p)&=i\phi(N)\sin\left(\frac{2 \pi p}{N} \right). \label{eq.sin-2}
\end{align}
We introduce the sets
\begin{equation}
\mathcal{I}_\pm(N) = \{\text{$\chi$ mod $N$ : $\chi$ imprimitive, $\chi \neq \chi_0$, $\chi(-1) = \pm 1$} \},
\end{equation}
so that equations~\eqref{eq.cos-2} and~\eqref{eq.sin-2} may be rewritten as follows:
\begin{equation}\label{eq.imprim-cos}
\sum_{\chi \in \mathcal D_+(N)} \tau(\overline{\chi}) \chi(p)
=1+ \phi(N)\cos \left (
\frac{2 \pi p} N \right )- \sum_{\chi \in \mathcal{I}_+(N)} \tau(\overline{\chi}) \chi(p),
\end{equation}
\begin{equation}\label{eq.imprim-sin}
\sum_{\chi \in \mathcal D_-(N)} \tau(\overline{\chi}) \chi(p)
=i\phi(N)\sin \left (
\frac{2 \pi p} N \right )- \sum_{\chi \in \mathcal{I}_-(N)} \tau(\overline{\chi}) \chi(p).
\end{equation}
Applying equation~\eqref{eq.1/G} to equations~\eqref{eq.imprim-cos} and~\eqref{eq.imprim-sin}, we deduce
\begin{equation}\label{eq.imprim-cos-2}
\sum_{\chi \in \mathcal D_+(N)} \frac{\chi(p)}{\tau(\chi)}
=\frac{1}{N}+ \frac{\phi(N)}{N}\cos \left (
\frac{2 \pi p} N \right )- \frac{1}{N}\sum_{\chi \in \mathcal{I}_+(N)} \tau(\overline{\chi}) \chi(p),
\end{equation}
\begin{equation}\label{eq.imprim-sin-2}
\sum_{\chi \in \mathcal D_-(N)} \frac{\chi(p)}{\tau(\chi)}
=\frac{-i\phi(N)}{N}\sin \left (
\frac{2 \pi p} N \right )+ \frac{1}{N}\sum_{\chi \in \mathcal{I}_-(N)} \tau(\overline{\chi}) \chi(p).
\end{equation}
In order to recreate the proof of Theorem~\ref{thm.main} for composite conductors, equations~\eqref{eq.imprim-cos-2} and~\eqref{eq.imprim-sin-2} suggest that we need to analyse sums of imprimitive characters. 
The following lemma will be useful for that purpose.
\begin{lemma}\label{lem.im-mob}
If an imprimitive character $\chi$ mod $N$ is induced by the primitive character $\chi_1$ mod $N_1$, then,  we have
\begin{equation}\label{eq.imprimGauss}
\tau(\overline{\chi})=\mu \left({\frac {N}{N_{1}}}\right)\overline{\chi_{1}\left({\frac {N}{N_{1}}}\right)}G\left(\overline{\chi _{1}}\right), 
\end{equation}
where $\mu(n)$ is the M\"obius function as before.
\end{lemma}

\begin{proof}
\cite[Lemma 3.1]{IK}.
\end{proof}
Inspired by equations~\eqref{eq.imprim-cos-2} and~\eqref{eq.imprim-sin-2}, we introduce the following functions:
\begin{equation}\label{eq.Ilavanish_dyadic}
E_{\pm}(y,X,c)=\frac{1}{X}\sum_{\substack{N \in[X, cX] \\ N \not\equiv 2 \bmod 4 }} \frac{1}{N} \sum_{\mathcal{I}_\pm(N)} \tau(\overline{\chi}) \chi(\lceil yX \prceil),
\end{equation}
\begin{equation}\label{eq.Ilavanish_short}
\widetilde{E}_{\pm}(y,X,\delta)=\frac{1}{X^\delta}\sum_{\substack{N \in[X, X+X^\delta] \\ N \not\equiv 2\bmod 4}} \frac{1}{N} \sum_{\mathcal{I}_\pm(N)} \tau(\overline{\chi}) \chi(\lceil yX \prceil).
\end{equation}
One sees that it is natural to investigate:
\begin{equation}\label{eq.Murm}
T_{\pm}(y,X,c)=Q_{\pm}(y,X,c)\pm E_{\pm}(y,X,c), 
\end{equation} 
\begin{equation}\label{eq.Murm2}
\widetilde{T}_{\pm}(y,X,\delta)=\widetilde{Q}_{\pm}(y,X,\delta)\pm\widetilde{E}_{\pm}(y,X,\delta).
\end{equation}
\begin{remark}
Figure \ref{fig:idft_dyadic_lavg} (resp. Figure \ref{fig:idft_short_lavg}) suggest that $|T_\pm(y,X,c)|$ and $|Q_\pm(y,X,c)|$ (resp.  $|\widetilde{T}_\pm(y,X,c)|$ and $|\widetilde{Q}_\pm(y,X,\delta)|$) are significantly larger than $|E_\pm(y, X, \delta)|$ (resp. $|\widetilde{E}_\pm(y,X,\delta)|$).
Since there is a canonical bijection between Dirichlet characters mod $N$ and primitive Dirichlet characters with conductor dividing $N$, $E_\pm(y, X, \delta)$ (resp. $\widetilde{E}_\pm(y,X,\delta)$) reduces to a sum over primitive characters with conductor dividing $N$. 
Using Lemma~\ref{lem.im-mob}, we see that this introduces a M\"obius factor. Consequently, we expect that this term is smaller due to additional cancellation. 
\end{remark}
Another complexity arising from equations~\eqref{eq.imprim-cos-2} and~\eqref{eq.imprim-sin-2} is the need to understand murmuration-type limits for $\phi(N)/N$, for which the following lemma will be useful.

\begin{lemma}\label{lem.phi_order}
If $a\in\mathbb{R}_{> 0}$ and $b \in (0,1]$, then
\begin{equation}\label{eq.5/pi2}
\lim_{X \to \infty} \frac{1}{aX^b} \sum_{\substack{N \in [X, X + aX^b] \\ N \not\equiv 2\bmod 4}} \frac{\phi(N)}{N}  
=\frac{5}{\pi^2}.
\end{equation}
\end{lemma}
\begin{proof}
It is known that
\begin{equation}\label{eq.order_phi_error}
\sum_{\substack{0<N \leq X\\N\in\mathbb{Z}}} \frac{\phi(N)}{N} = \frac{6}{\pi^2}X+O\left((\log X)^{2/3}(\log \log X)^{4/3}\right), 
\end{equation}
from which it follows that
\begin{equation}\label{eq.order_phi}
\lim_{X \to \infty} \frac{1}{X}\sum_{\substack{0<N \leq X\\N\in\mathbb{Z}}} \frac{\phi(N)}{N} = \frac{6}{\pi^2} 
\end{equation}
(cf. \cite[equation~(1.74)]{IK}).
Similarly, according to \cite{N75}, we have that
\begin{equation}
    \lim_{X \to \infty} \frac{1}{X} \sum_{\substack{0< N \leq X \\ N \in \mathbb{Z}}}  \frac{\phi(2N+1)}{2N+1} = \frac{8}{\pi^2}.\label{eq.order_phi_odd}
\end{equation}
Using equation~\eqref{eq.order_phi_odd} and the identity $\phi(4N+2) = \phi(2N+1)$, we compute:
\begin{equation}
\lim_{X \to \infty} \frac{1}{X} \sum_{\substack{0<N \leq X \\ N \equiv 2\bmod 4}}\frac{\phi(N)}{N} 
= \lim_{X \to \infty} \frac{1}{4X} \sum_{\substack{0<N \leq X\\N\in\mathbb{Z}}} \frac{\phi(4N+2)}{4N+2} = \frac{1}{8} \lim_{X \to \infty} \frac{1}{X} \sum_{\substack{0<N \leq X\\N\in\mathbb{Z}}} \frac{\phi(2N+1)}{2N+1} = \frac{1}{\pi^2}. \label{eq.order_phi_2(4)}
\end{equation}
Subtracting equation~\eqref{eq.order_phi_2(4)} from equation~\eqref{eq.order_phi}, and noting the error term in equation~\eqref{eq.order_phi_error},  we conclude 
\begin{equation}\label{eq.N<X5pi2}
\lim_{X \to \infty} \frac{1}{X} \sum_{\substack{0<N \leq X \\ N \not\equiv 2 \bmod 4}} \frac{\phi(N)}{N} 
=\frac{5}{\pi^2}.
\end{equation}
Equation~\eqref{eq.N<X5pi2} implies that 
\begin{equation}
    \sum_{\substack{0<N \leq X \\ N \not\equiv 2 \bmod 4}} \frac{\phi(N)}{N} \sim \sum_{\substack{0<N \leq X \\ N \in \mathbb{Z}}} \frac{5}{\pi^2},
\end{equation}
from which equation~\eqref{eq.5/pi2} follows.
\end{proof}

\subsubsection{Geometric Intervals}
In this subsection, we will prove the following theorem, which is visualised in Figure~\ref{fig:idft_dyadic_lavg}.
\begin{theorem}\label{thm.dyadic}
If $c\in\mathbb{R}_{>1}$ and $y\in\mathbb{R}_{>0}$, then
\begin{equation}\label{eq.int}
\lim_{X \to \infty} T_{\pm}(y,X,c) =
\begin{cases}
\frac{5}{\pi^2}\int_1^c \cos\left(\frac{2\pi y}{x} \right)dx, & \text{ if }+, \\
-i\frac{5}{\pi^2}\int_1^c \sin\left(\frac{2\pi y}{x} \right)dx, & \text{ if }-.
\end{cases}
\end{equation}
\end{theorem}

\begin{figure}[h]
\centering
\includegraphics[width=0.7\textwidth]{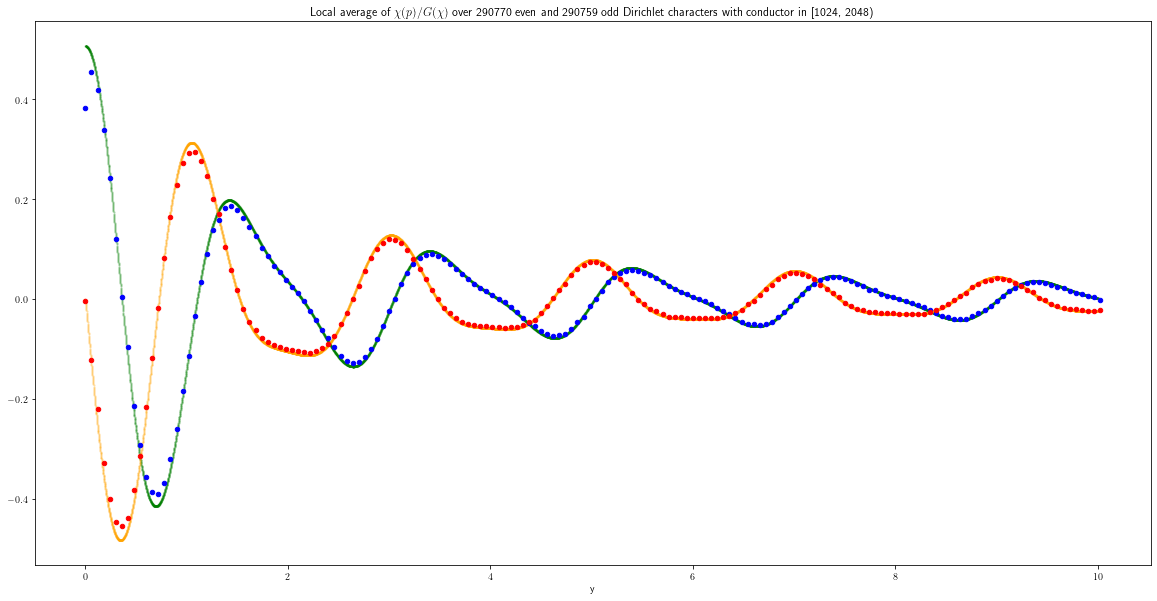}
\caption{\sf   Plot of $T_\pm(y, 1024, 2)$ for $0 \leq y \leq 10$ with $+$ in blue and (the imaginary part of) $-$ in red.  We also show $\frac{5}{\pi^2}\int_1^2 \cos \left(\frac{2\pi y}{x}\right)dx$ in green and $-\frac{5}{\pi^2}\int_1^2 \sin \left(\frac{2\pi y}{x}\right)dx$ in orange. }
\label{fig:idft_dyadic_lavg}
\end{figure}

\begin{proof}
We will prove the case of $T_+(y,X,c)$, and simply note that $T_-(y,X,c)$ is similar. 
Applying equations~\eqref{eq.1/N},~\eqref{eq.Mpm}, and~\eqref{eq.Ilavanish_dyadic} to equation~\eqref{eq.imprim-cos-2}, we deduce
\begin{equation}\label{eq.tobreak}
\begin{split}
\lim_{X \to \infty} T_{+}(y,X,\delta) &= \lim_{X \to \infty} \frac{1}{X} \sum_{\substack{N \in [X, cX] \\ N \not\equiv 2\bmod 4}} \frac{\phi(N)}{N}\cos\left(\frac{2\pi yX}{N}\right)-\lim_{X\rightarrow\infty}E_+(y,X,\delta)\\
&=\lim_{X \to \infty} \frac{1}{X} \sum_{i=1}^{n} \sum_{\substack{N \in I_i \\ N \not\equiv 2 \bmod 4}} \frac{\phi(N)}{N} \cos\left(\frac{2\pi y X}{N}\right)-\lim_{X\rightarrow\infty}E_+(y,X,\delta),
\end{split}
\end{equation}
where, for each $X$, we put $n=\left \lceil \sqrt{X} \right \rceil$ and, for $i\in\left\{1,\dots,n\right\}$, we write
\begin{equation}
I_i = \left[X + \frac{i-1}{n}(c-1)X, X + \frac{i}{n}(c-1)X \right).
\end{equation}
Fix $\gamma \in (0, 1]$ and, for each $X$, choose $i = \lceil \gamma n\rceil\in\{1,\dots,n\}$. 
We have
\begin{equation}\label{eq.gamma}
\lim_{X \to \infty} \frac{i-1}{n} = \lim_{X \to \infty }\frac{i}{n} = \gamma.
\end{equation}
For $N \in I_i$, equation~\eqref{eq.gamma} implies that
\begin{equation}\label{eq.another-squeeze}
\frac{1}{1+\gamma(c-1)} = \lim_{X \to \infty} \frac{X}{X+i(c-1)X/n} \leq \lim_{X \to \infty} \frac{X}{N} \leq \lim_{X \to \infty}\frac{X}{X+(i-1)(c-1)X/n} = \frac{1}{1+\gamma(c-1)}.
\end{equation}
Combining equations~\eqref{eq.Murm}, ~\eqref{eq.tobreak}, and ~\eqref{eq.another-squeeze}, we deduce
\begin{equation}\label{eq.prelemma}
\lim_{X \to \infty} T_{+}(y,X,\delta)  =  \lim_{X \to \infty}\frac{c-1}{\sqrt{X}} \sum_{i=1}^n  \cos\left(\frac{2\pi y}{1+\gamma(c-1)}\right) \frac{1}{(c-1)\sqrt{X}} \sum_{\substack{N \in I_i \\ N \not\equiv 2 \bmod 4}} \frac{\phi(N)}{N}.
\end{equation}
Using Lemma~\ref{lem.phi_order} and equation~\eqref{eq.prelemma}, we are led to
\begin{equation}\label{eq.RSfor5/pi2}
\lim_{X \to \infty} T_{+}(y,X,\delta)  = \lim_{n \to \infty} \frac{c-1}{n} \sum_{i=1}^n   \frac{5}{\pi^2}\cos\left(\frac{2\pi y}{1+i(c-1)/n}\right).
\end{equation}
Now equation~\eqref{eq.int} follows upon recognising equation~\eqref{eq.RSfor5/pi2} as a Riemann sum.
\end{proof}

\subsubsection{Short intervals}
We first prove the following theorem, which is visualised in Figure~\ref{fig:idft_short_lavg}.
\begin{theorem}
If $\delta\in(0,1)$ and $y\in\mathbb{R}_{>0}$, then
\begin{equation}\label{eq.moreconductors}
\lim_{X\rightarrow\infty}\widetilde{T}_{\pm}(y,X,\delta)= \begin{cases}
\frac{5}{\pi^2}\cos(2\pi y), & \text{ if $+$},\\
-i\frac{5}{\pi^2}\sin(2\pi y), & \text{ if $-$}.
\end{cases}
\end{equation}
\end{theorem}
\begin{figure}[h]
\centering
\includegraphics[width=0.7\textwidth]{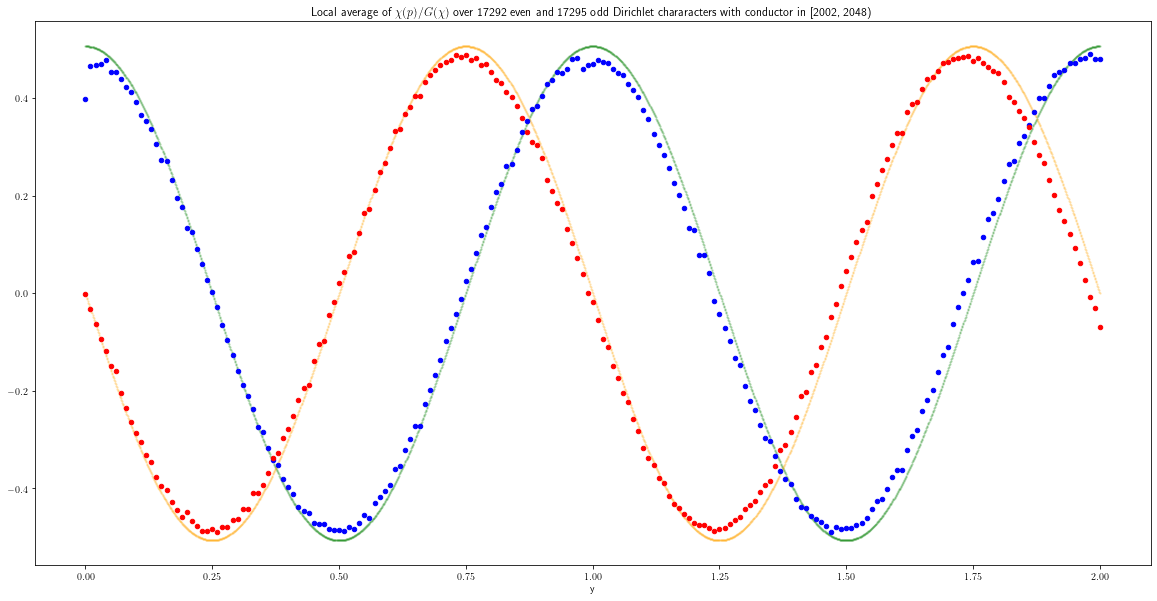}
\caption{\sf Plot of $\widetilde{T}_\pm(y, 2002, 0.51)$ for $0 \leq y \leq 2$ with $+$ in blue and (imaginary part of) $-$ in red.  We also show $\frac{5}{\pi^2}\cos(2\pi y)$ in green and $\frac{5}{\pi^2}\sin(2\pi y)$ in orange.}
\label{fig:idft_short_lavg}
\end{figure}

\begin{proof}
We will prove the case of $\widetilde{T}_+(y,X,\delta)$ and simply note that $\widetilde{T}_-(y,X,\delta)$ is similar. Mimicking the proof of equation~\eqref{eq.LAthm-1} leads to 
\begin{equation}\label{eq.limTtild}
    \lim_{X \to \infty} \widetilde{T}_+(y,X,\delta) = \lim_{X \to \infty} \frac{\cos(2\pi y)}{X^\delta} \sum_{\substack{N \in [X, X+X^\delta] \\ N \not\equiv 2 \text{ mod } 4}} \frac{\phi(N)}{N}. 
\end{equation}
In light of equation~\eqref{eq.limTtild}, the result follows from Lemma~\ref{lem.phi_order}.
\end{proof}
We next obtain an extension of equation~\eqref{eq.LAthm-1} by considering a set of special conductors specified as follows.  
Let $S$ denote the set of positive integers that are not congruent to $2$ mod $4$ and are either prime or squarefull\footnote{A positive integer is {\it squarefull} if all its prime factors exponents are at least $2$.}. 
By equation~\eqref{eq.imprimGauss},
this is precisely the set $S$ of integers such that, if $N \in S$, then
\begin{equation} \label{eq.sum_imprim}
    \sum_{\mathcal{I}_\pm(N)} \tau(\overline{\chi}) \chi(p) = 0, \ \ (N\in S).
\end{equation}
Using equations~\eqref{eq.1/N} and~\eqref{eq.sum_imprim}, we deduce that, for $N \in S$, equation~\eqref{eq.imprim-cos-2} reduces to 
\begin{equation} \label{eq.special_cos}
    \sum_{\chi \in \mathcal D_\pm(N)} \frac{\chi(p)}{\tau(\chi)}  = \frac{\phi(N)}{N}\cos\left(\frac{2\pi p}{N} \right), \ \ (N\in S).
\end{equation}
Now define 
\begin{equation}\label{eq.tau}
    f(X) = \sum_{\substack{N \leq X \\ N \in S}} \frac{\phi(N)}{N},
\end{equation}
and consider
\begin{equation}
    \widetilde{Q}_{\pm}^S(y, X, \delta) = \frac{1}{f(X+X^\delta) - f(X)} \sum_{\substack{N \in [X, X+X^\delta] \\ N \in S}} \sum_{\chi \in \mathcal{D}_\pm(N)} \frac{\chi(\lceil yX\rceil^p)}{\tau(\chi)}. 
\end{equation}
This leads to the following corollary.
\begin{corollary}
Under the Riemann hypothesis,
if $\delta\in(\frac12,1)$ and $y\in\mathbb{R}_{>0}$, then
\begin{equation}\label{eq.moreconductors-1}
\lim_{X\rightarrow\infty}\widetilde{Q}^S_{\pm}(y,X,\delta)= \begin{cases}
\cos(2\pi y), & \text{ if $+$},\\
-i\sin(2\pi y), & \text{ if $-$}.
\end{cases}
\end{equation}
\end{corollary}
\begin{proof}
We prove the case of $\widetilde{Q}_+^S(y,X,\delta)$ and simply note that $\widetilde{Q}_-^S(y,X,\delta)$ is similar. Using equation~\eqref{eq.special_cos} and mimicking the proof of equation~\eqref{eq.LAthm-1} yields 
\begin{equation}\label{eq.togeneralize}
\lim_{X \to \infty} \widetilde{Q}^S_{+}(y,X,\delta)= \lim_{X \to \infty} \frac{\cos(2\pi y)}{f(X+X^\delta) - f(X)}\sum_{\substack{N \in [X,X+X^\delta] \\ N \in S}} \frac{\phi(N)}{N},
\end{equation}
from which the result follows by equation~\eqref{eq.tau}. 
\end{proof}

\subsection{Zubrilina density for $\left(\frac{d}{\cdot}\right)$} \label{sec-zd}
Using the techniques from Section~\ref{s:realp}, we may investigate the following variation of equation~\eqref{eq.Mp-1}.
Assuming that $yX >2$, we have:
\begin{align}\label{eq.Mp-2}
\begin{split}
M^\dag_{\Phi}(y,X,\delta)&=\frac{\log X}{X^{1+\delta}}\sum_{\substack{p\in[yX,yX+X^{\delta}] \\ p \text{ prime}}}\left(\sum_{\substack{d\in \mathcal{G}\\d\equiv1\bmod 4}}\Phi\left(\frac{d}{X}\right)\chi_d(p)+\sum_{\substack{d\in \mathcal{G}\\d\equiv3\bmod 4}}\Phi\left(\frac{d}{X}\right)\chi_{4d}(p)\right)\sqrt{p}\\
&=\frac{\log X}{X^{1+\delta}}\sum_{\substack{p\in[yX,yX+X^{\delta}] \\ p \text{ prime}}}\sum_{\substack{d\in \mathcal{G}}}\Phi\left(\frac{d}{X}\right)\left(\frac{d}{p}\right)\sqrt{p}.
\end{split}
\end{align}
We note that $M^{\dag}_{\Phi}(y,X,\delta)$ involves $\left(\frac{d}{\cdot}\right)$, whereas $M_\Phi(y,X,\delta)$ involves $\left(\frac{8d}{\cdot}\right)$.
For plots of the function $M^{\dag}_{\Phi}(y,X,\delta)$, see Figure~\ref{fig:quad_murm-1}. 
\begin{figure}[h]
\centering
\includegraphics[width=0.8\textwidth]{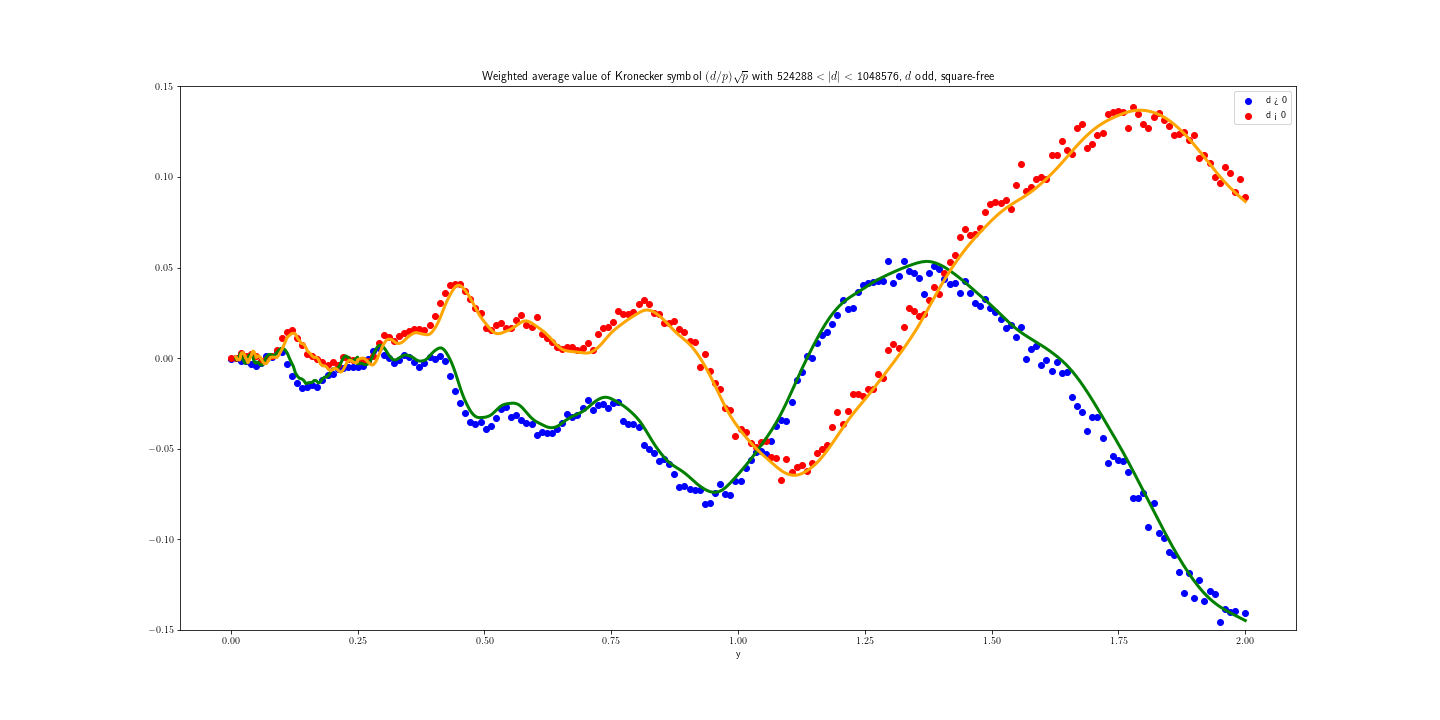}
\caption{\sf Let \[ \hspace*{-2.5cm} \Phi_+(x) =  \mathbbm{1}_{(1,2)}(x)\exp  (\tfrac{-1}{1-4(x-1.5)^2}  ), \ \  \Phi_-(x) =  \mathbbm{1}_{(-2,-1)}(x)\exp (\tfrac{-1}{1-4(-x-1.5)^2} ).\] 
 We plot  $M^{\dag}_{\Phi_\pm}(y, 2^{19}, 2/3)$ for $y \in [0,2]$ with $+$ in blue (resp. $-$ in red), and the right hand side of equation~\eqref{eq.MDMdag} in green (resp. orange). 
}
\label{fig:quad_murm-1}
\end{figure}
In this section, we calculate the following limit which yields Zubrilina density associated to $M^{\dag}_{\Phi}(y,X,\delta)$:
\begin{corollary} \label{cor-second}
Fix $y\in\mathbb{R}_{>0}$.
If $\delta\in (\frac34,1)$ and $\Phi \geq 0$ is a smooth Schwartz function with compact support, then, assuming the Generalized Riemann hypothesis, we have
 \begin{equation}\label{eq.MDMdag}
M^\dag_\Phi(y, \delta) \seteq    \lim_{X \to \infty}M^\dag_{\Phi}(y,X,\delta) = \frac{1}{2}\sum_{\substack{a =1 \\ (a, 2) = 1}}^\infty \frac{\mu(a)}{a^2} \sum_{m =1}^\infty \widetilde{\Phi}\left(\frac{m^2}{a^2y} \right).
\end{equation}   
\end{corollary}
\begin{proof}
Applying Lemma~\ref{lem.Sound-lemma} to equation~\eqref{eq.Mp-2}, we get
\begin{equation}
    M^\dag_{\Phi,A}(y,X,\delta) =\frac{\log X}{X^\delta} \sum_{\substack{p\in[yX,yX+X^{\delta}] \\ p \text{ prime}}}  \sum_{\substack{0<a \leq A \\ (a, 2p) = 1}} \frac{\mu(a)}{2a^2} \sum_{k \in \mathbb{Z}}  (-1)^k \left( \frac{2k}{p}\right)\widetilde{\Phi}\left(\frac{kX}{2a^2 p} \right).\label{eq.MpAldsp-1}
\end{equation}
Following the argument in Section~\ref{eq.anoMPA}, we observe that terms corresponding to $2k \neq \square$ vanish in equation~\eqref{eq.MpAldsp-1}. 
On the other hand, if $2k = \square$, then writing $k = 2m^2$ yields
\begin{equation}\label{eq.MpAldsp-2}
       \sum_{\substack{k \in \mathbb{Z}\\ 2k = \square}}  (-1)^k \left( \frac{2k}{p}\right)\widetilde{\Phi}\left(\frac{kX}{2a^2 p} \right) = \sum_{m=1}^\infty  (-1)^{2m^2} \left( \frac{2m}{p}\right)^2\widetilde{\Phi}\left(\frac{m^2X}{a^2 p} \right) = \sum_{\substack{m=1 \\ (2m, p)=1}}^\infty  \widetilde{\Phi}\left(\frac{m^2X}{a^2 p} \right).
\end{equation}
Comparing equation~\eqref{eq.MpAldsp-2} with equation~\eqref{eq:M_A_squres-e}, we notice the following simplification to what remains of the argument from Section~\ref{eq.anoMPA}.
Namely, we do not need to introduce $H_a(w)$ since $(-1)^m$ is missing in the final expression of equation~\eqref{eq.MpAldsp-2}. In fact, it is enough to use $  H(w) = \widetilde{\Phi}(w^2)$. With this modification, we finish the proof by mimicking Section~\ref{eq.anoMPA}. 
\end{proof}
Unfolding the function $\widetilde{\Phi}$, we may recover the analogue of equation~\eqref{eq.unfolding} for $M_{\Phi}^{\dag}(y,X,\delta)$.
Subsequently, one may compute the Zubrilina density for the family $\left\{\left(\frac{d}{\cdot}\right):d\in\mathcal{G}\right\}$.
Following the proof in Section~\ref{s:1ld}, we also obtain the following analogue of Corollary \ref{c:1ld}. 
\begin{corollary}\label{c:1ld-1}
Let $\Phi$ be a Schwartz function with compact support and let $\delta \in  (\frac34,1)$. Assuming the Generalized Riemann hypothesis, we have
\begin{equation}\label{eq.1leveldensity1-1}
\lim_{y\rightarrow0^+} M^{\dag}_{\Phi}(y, \delta)
=0, \quad \text{ and } \quad \lim_{y\rightarrow \infty} M^{\dag}_{\Phi}(y, \delta) =-\frac 2 {\pi^2} \widetilde \Phi(0),
\end{equation}
where $M^\dag_\Phi(y, \delta)$ is defined in \eqref{eq.MDMdag}.
\end{corollary}

\begin{proof}
Recall from the proof of Corollary \ref{cor-second} that, when modifying the argument from Section~\ref{eq.anoMPA}, we do not need to introduce $H_a(w)$. 
In particular, we see that the equation corresponding to   equation~\eqref{eq:m_phi_divsum} is given by 
    \begin{equation}
        M_\Phi^\dag(y, \delta) = -\frac{2}{\pi^2}\widetilde\Phi(0) + \frac{\sqrt{y}} 2 \sum_{\substack{n=1}}^\infty \left(\sum_{\substack{a \mid n \\ (a,2)=1}} \frac{\mu(a)}{a} \right)\widehat{H}\left(n\sqrt{y} \right).
    \end{equation}
   For all $n \in \mathbb{Z}_{\geq 1}$, set 
   \[b_n^\dag = \displaystyle{ \sum_{\substack{a \mid n \\ (a,2)=1}}} \frac{\mu(a)}{a}.\] In particular, we have $b_{2^k}^\dag = 1$. We also set $B^\dag(s) = \sum_{n=1}^\infty b_n^\dag n^{-s}$, so that $\mathrm{Res}_{s = 1}B^\dag(s) =  8/{\pi^2}$. Using the same argument as in Section \ref{s:1ld}, we find:
    \begin{equation}
        \frac{\sqrt{y}}2  \sum_{n=1}^\infty \left(\sum_{\substack{a \mid n \\ (a,2) = 1}} \frac{\mu(a)}{a}\right) \widehat{H}(n\sqrt{y}) \sim \frac{\sqrt{y}} 2 \frac{4\widetilde{\Phi}(0)}{\pi^2 \sqrt{y}} = \frac{2}{\pi^2}\widetilde{\Phi}(0),
    \end{equation}
    and we deduce the limit as $y \to 0^+$. The limit as $y \to \infty$  follows from Lemma \ref{lem:decay} as before.
\end{proof}


\begin{thebibliography}{CERP}

\bibitem[A76]{A76}
T. Apostol,
\newblock {\it Introduction to Analytic Number Theory},
\newblock {Undergraduate Texts in Mathematics, Springer-Verlag} {(1976)}.

\bibitem[BHP]{BHP}
R. C. Baker, G. Harman and J. Pintz,
\newblock {\it The difference between consecutive primes {II}},
\newblock {Proc. London Math. Soc.} {\bf 83} {(2001)}, no. 3, {532--562}.

\bibitem[B98]{Bump}
D.~Bump, 
\newblock{\it Automorphic forms and representations}, 
\newblock Cambridge University Press (1998).

\bibitem[CFS]{CFS}
J.~B.~Conrey, D.~W.~Farmer, and K.~Soundararajan, 
\newblock{\it Transition Mean Values of Real Characters}, 
\newblock Journal of Number Theory 82 (2000): 109-120.


\bibitem[C23]{C23}
A.~Cowan, 
\newblock{\it Murmurations and explicit formulas}, 
\newblock arXiv:2306.10425 (2023).



\bibitem[GS07]{GS07}
A. Granville and K. Soundararajan, 
\newblock{\it Large character sums: pretenious characters and P\'olya--Vinogradov theorem}, 
\newblock J. Amer. Math. Soc. {\bf 20} (2007), no. 2, 357--384.

\bibitem[HLO1]{HLO1}
Y.-H.~He, K.-H.~Lee and T.~Oliver,
\newblock{\it Machine learning the Sato--Tate conjecture}, 
\newblock J. Symbolic Comput. {\bf 111} (2022), 61--72.

\bibitem[HLO2]{HLO2}
Y.-H.~He, K.-H.~Lee and T.~Oliver,
\newblock{\it Machine learning number fields}, 
\newblock Mathematics, Computation and Geometry of Data {\bf 2} (2022), 49--66.

\bibitem[HLO3]{HLO3}
Y.-H.~He, K.-H.~Lee and T.~Oliver,
\newblock{\it Machine learning invariants of arithmetic curves}, 
\newblock J. Symbolic Comput. {\bf 115} (2023), 478--491.

\bibitem[HLOP]{HLOP}
Y.-H.~He, K.-H.~Lee, T.~Oliver, and A. Pozdnyakov,
\newblock{\it Murmurations of elliptic curves}, 
\newblock arXiv:2204.10140 (2022).

\bibitem[HLOPS]{HLOPS}
Y.-H.~He, K.-H.~Lee, T.~Oliver, A. Pozdnyakov, and A. Sutherland,
\newblock{\it Murmurations of $L$-functions}, 
\newblock in preparation. 

\bibitem[IK04]{IK}
H.~Iwaniec and E.~Kowalski, 
\newblock{\it Analytic number theory}, 
\newblock American Mathematical Soc. (2004).

\bibitem[J10]{J10}
G. J. O. Jameson,
\newblock{\it Even and odd square-free numbers}, Mathematical Gazette {\bf 94(529)} (2010), 123--127.

\bibitem[J73]{J73}
H. Jager,
\newblock{\it On the number of Dirichlet characters with modulus not exceeding x}, Indagationes Mathematicae (Proceedings) {\bf 76} (1973), no. 5, 452--455.

\bibitem[N75]{N75}
J.~E. Nymann,
\newblock{\it On the probability that $k$ positive integers are relatively prime II},
Journal of Number Theory {\bf 7} (1975), no. 4, 406--412.

\bibitem[RS94]{RS}
M.~Rubinstein and P.~Sarnak,
\newblock{\it Chebyshev's bias}, Experimental Mathematics {\bf 3} (1994), no. 3, 173-197.

\bibitem[S23i]{S23i}
P.~Sarnak, \newblock{\it Root numbers and murmurations}, ICERM Workshop, July 7, 2023.

\bibitem[S23ii]{S23ii}
P.~Sarnak, \newblock{\it Letter to Sutherland and Zubrilina}, August, 2023.


\bibitem[S76]{S76}
L. Schoenfeld, \newblock{\it Sharper Bounds for the Chebyshev Functions $\theta(x)$ and $\psi(x)$. II.} Mathematics of Computation {\bf 30} (1976), no. 134, 337--360. 

\bibitem[S00]{S00}
K. Soundararajan, \newblock{\it Nonvanishing of quadratic Dirichlet $L$-functions at s=1/2}, Ann. Math. {\bf 152} (2000), no.2, 447--488. 

\bibitem[T87]{T87}
 E. C. Titchmarsh, \newblock{\it The Theory of the Riemann Zeta-Function. Second Edition}, Oxford University Press (1986).

\bibitem[Z23]{Z23} N.~Zubrilina, \newblock{\it Murmurations}, arXiv:2310.07681 (2023).

\end{thebibliography}
\end{document}